\documentclass[12pt]{amsart}
\usepackage{graphicx} 
\usepackage{caption}
\usepackage{subcaption}
\usepackage{url} 
\usepackage{hyperref}
\usepackage{amsfonts,amsthm,latexsym,amsmath,amssymb,amscd,amsmath,epsf} 
\usepackage{tikz}
\usepackage{verbatim}
\usetikzlibrary{arrows,shapes}

\theoremstyle{plain}
\newtheorem{theorem}{\bf Theorem}[section]
\newtheorem{lemma}[theorem]{\bf Lemma}
\newtheorem{proposition}[theorem]{\bf Proposition}
\newtheorem{corollary}[theorem]{\bf Corollary}

\newtheorem{question}[theorem]{\bf Question}

\theoremstyle{definition}

\newtheorem{example}[theorem]{Example}

\theoremstyle{remark}
\newtheorem{remark}[theorem]{Remark}

\numberwithin{equation}{section}

\makeindex

\newcommand{\xx}{\mathbf{x}}

\newcommand{\QQ}{\mathbb{Q}}
\newcommand{\PP}{\mathbb{P}}
\newcommand{\NN}{\mathbb{N}}

\newcommand{\B}{\mathbb{B}}
\newcommand{\WC}{\mathbb{WCOMP}}

\newcommand{\T}{\mathcal{T}}
\newcommand{\W}{\mathcal{W}}

\newcommand{\Q}{\mathcal{Q}}
\newcommand{\A}{\mathcal{A}}
\renewcommand{\L}{\mathcal{L}}

\newcommand{\LL}{\mathbf{\varLambda}}

\newcommand{\sym}{\mathfrak{S}}
\newcommand{\lie}{\mathcal{L}ie}

\newcommand{\clr}{\mathbf{color}}

\newcommand{\BT}{\mathcal{BT}}
\newcommand{\comb}{\textsf{Comb}}

\newcommand{\lyn}{\textsf{Lyn}}

\newcommand{\nor}{\textsf{Nor}}
\newcommand{\NC}{\textsf{NC}}

\newcommand{\lyndonlambda}{\lambda^{\textsf{Lyn}}}
\newcommand{\comblambda}{\lambda^{\textsf{Comb}}}
\newcommand{\aalambda}{\lambda^{\textsf{AA}}}
\newcommand{\tnlambda}{\lambda^{\textsf{TN}}}
\newcommand{\dalambda}{\lambda^{\textsf{DA}}}
\newcommand{\inlambda}{\lambda^{\textsf{IN}}}

\def\newop#1{\expandafter\def\csname #1\endcsname{\mathop{\rm #1}\nolimits}}

\newop{diag}
\newop{Tor}
\newop{supp}
\newop{per}
\newop{rank}
\newop{ker}
\newop{cl}
\newop{step}
\newop{col}
\newop{red}
\newop{blue}
\newop{length}
\newop{sgn}
\newop{val}
\newop{root}
\newop{des}
\newop{asc}
\newop{pla}
\newop{couple}
\newop{free}
\newop{l}
\newop{init}
\newop{final}
\newop{rgp}
\newop{rk}
\newop{wcomp}
\newop{comp}
\newop{bar}
\newop{last}
\newop{PF}
\newop{SP}

\title[]{A family of symmetric functions associated with Stirling permutations}
\author[R. S. Gonz\'alez D'Le\'on]{Rafael S. Gonz\'alez D'Le\'on}
\address{Escuela de Ciencias Exactas e Ingenier\'ia, Universidad Sergio Arboleda, 
Bogot\'a, Colombia}
\email{rafael.gonzalezl@usa.edu.co}
\urladdr{\url{http://dleon.combinatoria.co}}

\begin{document}
\begin{abstract}
We present exponential generating function analogues to two classical identities involving the 
ordinary generating function of the complete homogeneous symmetric functions. After a suitable 
specialization the new identities reduce to identities involving the first and second order 
Eulerian polynomials. The study of these identities led us to consider a family of symmetric 
functions associated with a class of permutations introduced by Gessel and 
Stanley, known in the literature as  Stirling permutations. In 
particular, we define certain type statistics on Stirling permutations that refine the statistics of 
descents, ascents and plateaux and we show that their refined versions are equidistributed, 
generalizing a result of B\'ona.
The definition of this family of symmetric functions extends to the generality of $r$-Stirling 
permutations. We discuss some occurrences of these symmetric functions in 
the cases of $r=1$ and $r=2$.
\end{abstract}

\maketitle
\section{Preliminaries and notation}\label{section:notation}
We denote by $\NN$ the set of nonnegative integers, $\PP$ the set of positive integers and $\QQ$ 
the set of rational numbers.  A \emph{weak 
composition} is an infinite sequence $\mu=(\mu_1,\mu_2,\dots)$ of numbers $\mu_i \in \NN$ such that 
its \emph{sum} $|\mu|:=\sum_i \mu_i$ is finite. If $|\mu|=n$ for some $n \in \NN$, we say  that 
$\mu$ is a \emph{weak composition of $n$}. We denote by $\wcomp$ the set of weak compositions and 
$\wcomp_n$ the set of weak compositions of $n$. An \emph{(integer) partition} $\lambda$ of $n$ 
(denoted $\lambda \vdash n$) is a weak composition of $n$ whose entries are nonincreasing, i.e., 
$\lambda=(\lambda_1 \ge \lambda_2 \ge \cdots)$. If $\nu \in  \wcomp$ is  obtained by permuting the 
entries 
of another $\mu \in \wcomp$ we say that $\nu$ is a \emph{rearrangement} of $\mu$. We denote the set 
of rearrangements of $\mu$ by $\wcomp_{\mu}$. Let $\xx= 
x_1,x_2,\dots$ be an infinite set of variables and throughout this document we 
denote $x^{\mu}:=\prod_i x_i^{\mu_i}$ for $\mu \in \wcomp$. We denote by $\Lambda=\Lambda_{\QQ}$ 
the ring of symmetric 
functions in $\xx$ with rational coefficients, that is, the ring of power series 
on 
$\xx$ of bounded degree that are invariant under permutation of the variables. Let
$h_{\lambda}(x)$ and $e_{\lambda}(x)$ respectively denote the complete homogeneous symmetric 
function and the elementary symmetric function indexed by a partition $\lambda$. It is known that 
for $n \ge 0 $, the sets $\{h_{\lambda}\,\mid\, \lambda \vdash n\}$ and 
$\{e_{\lambda}\,\mid\, \lambda \vdash n\}$ are bases 
for the $n$-th homogeneous graded component of $\Lambda_{\QQ}$.
For information not presented here regarding symmetric functions the reader 
could go to \cite{Macdonald1995} and \cite[Chapter 7]{Stanley1999}. 


For a sequence $(a_0,a_1,\dots)$ of elements in a ring $R$ (containing $\QQ$)  the \emph{ordinary 
generating function} (or \emph{o.g.f.}) of 
$(a_n)$ is the formal power series $\sum_{n\ge0}a_ny^n \in R[[y]]$ and the \emph{exponential 
generating function} (or \emph{e.g.f.}) of $(a_n)$ is the formal power series 
$\sum_{n\ge0}a_n\dfrac{y^n}{n!}\in R[[y]]$ (cf. \cite{Stanley2012}). 
In all the following $f^{-1}$ denotes the multiplicative inverse and $f^{\left\langle -1 
\right\rangle}$ denotes the compositional inverse of $f \in R[[y]]$ whenever any of these inverses 
exist.

\section{Introduction}\label{section:introduction}
We consider the ring $\Lambda[[y]]$ of power series in the variable $y$ with coefficients in 
$\Lambda$. 
The following two identities are classical results in the study of symmetric functions.

\begin{proposition}[cf. \cite{Macdonald1995} Equation 
(2.6)]\label{proposition:ordinarymultiplicative}
We have 
 \begin{align*}
   \left(\sum_{n\ge0}(-1)^{n}h_{n}(\xx)y^n\right )
^{-1}=\sum_{n\ge0}e_{n}(\xx)y^n.
 \end{align*}
\end{proposition}

\begin{proposition}[cf. \cite{Stanley1997}]\label{proposition:ordinarycompositional}
 We have
 \begin{align*}
   \left(\sum_{n\ge1}(-1)^{n-1}h_{n-1}(\xx)y^n\right )
^{\left\langle -1 \right\rangle}=\sum_{n\ge1}\omega \PF_{n-1}(\xx)y^n,
 \end{align*}
 where $\omega$ is the involution in $\Lambda$ defined by  $\omega(h_n(\xx))=e_n(\xx)$ and 
$\PF_{n-1}(\xx)$ is Garsia-Haiman's parking function symmetric function, see \cite{Haiman1994} and 
\cite{GarsiaHaiman1996}.
\end{proposition}

It is known that $\omega \PF_{n-1}(\xx)$ has positive 
coefficients when expressed in the elementary basis (see \cite{Haiman1994} and \cite{Stanley1997}), 
a property known as \emph{$e$-positivity}.
\begin{proposition}[c.f. \cite{Stanley1997}]For $n\ge 0$,
 $$\omega \PF_{n}(\xx)=\sum_{\pi \in \NC_n}e_{\lambda(\pi)}(\xx),$$
 where $\NC_n$ is the set of noncrossing partitions of $[n]$ and 
$\lambda(\pi)$ is the integer 
partition of $n$ whose parts are the sizes of the blocks of the set partition $\pi$.
\end{proposition}

A common 
feature of Propositions \ref{proposition:ordinarymultiplicative} and 
\ref{proposition:ordinarycompositional} is the e-positivity of the coefficients of the power series 
in the right-hand side.

In this work we find exponential generating function analogues to Propositions 
\ref{proposition:ordinarymultiplicative} and \ref{proposition:ordinarycompositional}. We prove the 
following theorems.

\begin{theorem}\label{theorem:exponentialmultiplicative}
 We have
 \begin{align}\label{equation:exponentialmultiplicativeinverse}
   \left(\sum_{n\ge0}(-1)^{n}h_{n}(\xx)\dfrac{y^n}{n!}\right )
^{-1}=\sum_{n\ge0}\sum_{\sigma \in \sym_n}e_{\lambda(\sigma)}(\xx)\dfrac{y^n}{n!},
 \end{align}
 where $\sym_n$ is the set of permutations of $[n]=:\{1,2,\dots,n\}$ and  
$\lambda(\sigma)$ is the consecutive ascending type 
of $\sigma \in \sym_n$ (defined in Section \ref{section:stirlingtype}) .
\end{theorem}

\begin{theorem}\label{theorem:exponentialcompositional}
 We have
 \begin{align}\label{equation:exponentialcompositionalinverse}
   \left(\sum_{n\ge1}(-1)^{n-1}h_{n-1}(\xx)\dfrac{y^n}{n!}\right )
^{\left\langle -1 \right\rangle}=\sum_{n\ge1}\sum_{\theta \in 
\Q_{n-1}}e_{\lambda(\theta)}(\xx)\dfrac{y^n}{n!},
 \end{align}
 where $\Q_n$ is the set of Stirling (multi)permutations of the multiset 
$[n]\sqcup[n]=:\{1,1,2,2,\dots,n,n\}$ (defined by Gessel and Stanley in
\cite{StanleyGessel1978}) and  
where $\lambda(\theta)$ is one of the types $\aalambda$, $\dalambda$, $\tnlambda$ and 
$\inlambda$ (defined in  Section \ref{section:stirlingtype}) for $\theta \in \Q_n$.
\end{theorem}

Theorem \ref{theorem:exponentialcompositional} was first derived by the author in \cite{Dleon2013} 
using poset topology techniques applied to a poset of partitions weighted by weak compositions. The 
coefficient of $\dfrac{y^n}{n!}$ in the power series of the right-hand side of equation 
(\ref{equation:exponentialcompositionalinverse}) is the generating function for the dimensions of 
the reduced (co)homology of the maximal intervals of the poset of weighted partitions and also for 
the dimensions of the multilinear components of the free
Lie algebra with multiple compatible brackets on $n$ generators.
Here we take a different approach that does not involve poset theoretic techniques. We provide a 
different combinatorial proof of Theorem \ref{theorem:exponentialcompositional} using an 
interpretation by B. Drake of the compositional inverse of an exponential generating function in 
\cite{Drake2008}. 

Drake's technique is a kind of combinatorial Lagrange inversion that 
can be used under certain conditions, namely when the power series is a generating function for a 
family of trees constructed out of basic building blocks and involves only quadratic instructions 
(two blocks at a time) on how to compose building blocks. There is a great deal of literature on 
Lagrange inversion and combinatorial Lagrange inversion. In particular, Joyal's theory of 
combinatorial species provides a combinatorial framework for Lagrange inversion with minimal 
assumed conditions. We refer the reader to the book of Bergeron, Labelle 
and Leroux \cite{BergeronLabelleLeroux1998} for an account on the subject. There exist also 
various generalizations of Lagrange inversion, see for example the $q$-Lagrange inversions 
studied by Garsia \cite{Garsia1981}, Gessel \cite{Gessel1980} and Garsia and Haiman 
\cite{GarsiaHaiman1996}.
The advantage of Drake's approach is that whenever the generating function satisfies certain 
conditions the interpretation for the compositional inverse becomes simple, uses the same type of  
combinatorial objects as the ones that are being counted and does not involve any cancellations due 
to signs. We remark here that, because of the independence of the symmetric functions 
$h_n(\xx)$ in the ring $\Lambda$, equation \eqref{equation:exponentialcompositionalinverse} can be 
seen as a Lagrange inversion formula for exponential power series. This work however is not about 
inversion of power series but instead is about investigating a family of symmetric functions that 
naturally appear when studying these inversion formulas.

In order to apply Drake's theorem we study a subset of the set of planar leaf-labeled binary trees 
that we call \emph{normalized}.  The normalization condition means that in any subtree the smallest 
label is in the leftmost leaf. This is equivalent to considering non-planar leaf-labeled binary 
trees 
(or phylogenetic trees) and the normalization condition is just a particular choice on how to draw 
these trees in the plane. Using Drake's technique we prove a version of Theorem 
\ref{theorem:exponentialcompositional} in terms of normalized trees (instead of Stirling 
permutations) and use this result and a bijection used in \cite{Dleon2013} between normalized 
trees and Stirling permutations (a bijection that appeared first in \cite{Dotsenko2012}) to derive 
Theorem \ref{theorem:exponentialcompositional}.

We then generalize the symmetric functions that appear as coefficients of the power series of the 
right-hand side of equation (\ref{equation:exponentialcompositionalinverse}) to the generality of 
the family $\Q_n(r)$ of $r$-Stirling permutations, where Stirling permutations correspond to the 
case $r=2$ and the classical permutations in the symmetric group to the case $r=1$. We consider the 
family of symmetric functions

\begin{align}\label{definition:s2_1}
 \SP_n^{(r)}(\xx)=\sum_{\theta \in \Q_n(r)}e_{\lambda(\theta)}(\xx),
\end{align}
where $\lambda(\theta)$ is any of various types of $\theta$ (defined in Section 
\ref{section:multipermutations}).

It turns out that the case $r=1$ is the family of symmetric functions that appear in the 
right hand side of equation (\ref{equation:exponentialmultiplicativeinverse}). In order to prove 
Theorem \ref{theorem:exponentialmultiplicative}, this time 
we use an interpretation of the multiplicative inverse of an exponential generating function that 
can be derived from a more general result discovered by Fr\"oberg \cite{Froberg1975}, 
Carlitz-Scoville-Vaughan \cite{CarlitzScovilleVaughan1976} and Gessel \cite{Gessel1977}. As in the 
case of Theorem \ref{theorem:exponentialcompositional} the author provides in 
\cite{Dleon2016} a second proof of Theorem 
\ref{theorem:exponentialmultiplicative} using poset topology techniques 
over a poset of subsets weighted by weak compositions. Some of the context of that proof is 
discussed in Section \ref{section:posetcohomology}. 

We note that after the simple specialization $e_i \mapsto t$ the function $\SP^{(r)}_n(\xx)$ 
reduces 
to 
$A^{(r)}_n(t)$, the $r$-th order Eulerian polynomial, that is the descent generating polynomial 
of 
the family of $r$-Stirling permutations (defined later). In the case $r=1$, $\SP^{(1)}_n(\xx)$ 
specializes to the classical Eulerian polynomial $A_n(t):=A^{(1)}_n(t)$, that is the descent 
generating 
polynomial of $\sym_n$, and equation 
(\ref{equation:exponentialmultiplicativeinverse}) specializes to the following classical result.

\begin{theorem}[Riordan \cite{Riordan1951}]\label{theorem:riordan} We have
 \begin{align*} 
   \dfrac{1-t}{1-te^{(1-t)y}}= \sum_{n\ge0}A_n(t)\dfrac{y^n}{n!}.
 \end{align*}
\end{theorem}

In the case $r=2$, we obtain the following analogous result.
\begin{theorem}\label{theorem:specializedexponentialcompositional}
 We have
 \begin{align*}
   \left(\dfrac{(1-t)y+(1-e^{y(1-t)})t}{(1-t)^2}\right )
^{\left\langle -1 \right\rangle}=\sum_{n\ge1}A^{(2)}_{n-1}(t)\dfrac{y^n}{n!}.
 \end{align*}
\end{theorem}

The paper is organized as follows: in Section \ref{section:binarytrees} we discuss Drake's 
interpretation of compositional inverses of exponential generating functions and use this 
interpretation to give a version of Theorem \ref{theorem:exponentialcompositional} in terms of the 
family of normalized labeled binary trees. In Section \ref{section:stirlingpermutations} we use a 
bijection between normalized labeled binary trees and Stirling permutations to prove Theorem 
\ref{theorem:exponentialcompositional}. We then consider the natural generalization 
(\ref{definition:s2_1}) of the symmetric 
functions that appear as coefficients in the right-hand side of equation 
(\ref{equation:exponentialcompositionalinverse}) and show in Section 
\ref{section:permutations} that in the base case they are precisely the family of symmetric 
functions that appear in the right-hand side of equation 
(\ref{equation:exponentialmultiplicativeinverse}). We discuss an interpretation of multiplicative 
inverses of exponential generating functions 
and use it to prove Theorem \ref{theorem:exponentialmultiplicative}. In Section 
\ref{section:specializations} we show that under a simple specialization Theorems 
\ref{theorem:exponentialmultiplicative} and \ref{theorem:exponentialcompositional} reduce to 
expressions involving first and second order Eulerian polynomials. Finally, in Section 
\ref{section:futuredirections} we briefly present other contexts where the 
symmetric functions $\SP^{(1)}_n(\xx)$ and $\SP^{(2)}_n(\xx)$ make an appearance. In  particular, 
these symmetric functions are the generating functions for the M\"obius invariants of the maximal 
intervals of two families of posets. We leave some open questions regarding the cases $r \ge 3$.

\section{Binary trees}\label{section:binarytrees}

A tree is a connected graph that has no loops or cycles. We say that a tree is rooted if 
one of its nodes is specially marked and called the \emph{root}. For two nodes $x$ and
$y$ on a rooted tree $T$, $x$ is said to be the \emph{parent} of $y$ (and $y$ the 
\emph{child} of $x$) if $x$ is the node that follows $y$ in the unique path from $y$ to the 
root. A node is called a \emph{leaf} if it has no children, otherwise is said to be 
\emph{internal}. A rooted tree $T$ is said to be \emph{planar} if for every internal node $x$ of 
$T$ the set of children of $x$ is totally ordered. A \emph{(leaf-)labeled (planar) tree} 
$(T,\sigma)$ is defined as a planar tree $T$ whose $j$th leaf from left to right 
has been labeled $\sigma(j)$, where $\sigma$ is a permutation. A rooted tree is said to be 
\emph{complete} if all internal nodes have the same number of children. A \emph{(complete planar) 
binary tree} is a planar rooted tree in which every internal node has exactly two children, a left 
and a right child. 
We denote by $\BT_n$  the set of leaf-labeled binary trees with $n$ leaves, these will play a 
relevant role in the following. See Figure \ref{fig:treecomposition} for some examples of labeled 
binary trees.

\subsection{Drake's interpretation of compositional inverse}\label{subsection:drakesinterpretation}

In \cite{Drake2008} B. Drake proposes an interesting interpretation of the compositional inverse
of an exponential generating function in terms of trees with allowed and forbidden links. This 
interpretation was also rediscovered by Dotsenko in \cite{Dotsenko2012}. 

Consider a set of rooted trees (either planar or not) whose leaves are all labeled with distinct 
positive integers. We also consider that the rooted trees in this set come together with some local
function, that we call a \emph{valency}, that allows us to extend consistently the 
labeling from the children to the parents. Hence, every internal node has been also assigned a 
label coming from the set of labels of the leaves. Formally a \emph{valency} is a 
recursively defined rule that assigns to each
node $x$ (internal or leaf)  of a leaf-labeled rooted tree $T$ a unique element $v(x)$ such that 
$v(x)\in \{v(y) \mid y \text{ a child of }x\}$ if $x$ is an internal node, or $v(x)=l$ if $x$ is a 
leaf with label $l$. The \emph{valency of a rooted tree} $v(T)$ is defined as the valency of its 
root.

\begin{example}
We can consider as an example the set of binary trees with distinct leaf-labels. 
A possible valency function is the one that assigns to a parent the smallest valency among the 
valencies of its children. In Figure \ref{fig:treecomposition} we illustrate some of these trees 
with the valency labels indicated near each node.
\end{example}

For two rooted trees $T_1$ and $T_2$ with label sets that only coincide
in one label that is both a leaf of $T_1$ and the valency of $T_2$, the 
\emph{composition} $T_1
\circ T_2$, is defined to be the tree obtained by deleting the common label from the leaf of $T_1$
and attaching the root of $T_2$ instead in its position (see Figure \ref{fig:treecomposition}). If 
the above condition for $T_1$ and 
$T_2$ is not
satisfied $T_1 \circ T_2$
is undefined. Note that composition is associative and so expressions like $T_1 \circ T_2 \circ
\cdots \circ T_k$ are well defined.

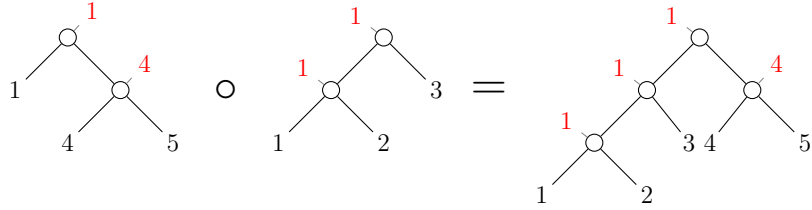
\begin{figure}[ht] 
  \begin{tikzpicture}[scale=0.7]
\begin{scope}[xshift=5cm,yshift=-1cm]

\tikzstyle{every node}=[draw,inner sep=1mm,scale=0.8]
    \draw [circle] (1,1)  node (i1)[pin={[red,pin distance=.05in]150:$1$}]{$$};
    \draw [circle] (2,2)  node (i2)[pin={[red,pin distance=.05in]150:$1$}]{$$};

\tikzstyle{every node}=[inner sep=1pt, minimum width=14pt,scale=0.8]

    \draw (0,0)  node (m){$1$};
    \draw (2,0)  node (l1){$2$};
    \draw (3,1)  node (l2){$3$};

    \draw (m) --  (i1) ;
    \draw (i1) --  (l1) ;
    \draw (i1) --  (i2) ;
    \draw (i2) --  (l2) ;
\end{scope}
\begin{scope}[xshift=-1cm,yshift=-1cm]

\tikzstyle{every node}=[draw,inner sep=1mm,scale=0.8]
    \draw [circle] (3,1)  node (i1)[pin={[red,pin distance=.05in]50:$4$}]{$$};
    \draw [circle] (2,2)  node (i2)[pin={[red,pin distance=.05in]50:$1$}]{$$};

\tikzstyle{every node}=[inner sep=1pt, minimum width=14pt,scale=0.8]

    \draw (2,0)  node (m){$4$};
    \draw (4,0)  node (l1){$5$};
    \draw (1,1)  node (l2){$1$};

    \draw (m) --  (i1) ;
    \draw (i1) --  (l1) ;
    \draw (i1) --  (i2) ;
    \draw (i2) --  (l2) ;
\end{scope}

\begin{scope}[xshift=10cm,yshift=-1cm]

\tikzstyle{every node}=[draw,inner sep=1mm,scale=0.8]
    \draw [circle] (1,0)  node (i1)[pin={[red,pin distance=.05in]150:$1$}]{$$};
    \draw [circle] (2,1)  node (i2)[pin={[red,pin distance=.05in]150:$1$}]{$$};
  \draw [circle] (4,1)  node (i4)[pin={[red,pin distance=.05in]50:$4$}]{$$};
    \draw [circle] (3,2)  node (i3)[pin={[red,pin distance=.05in]150:$1$}]{$$};

\tikzstyle{every node}=[inner sep=1pt, minimum width=14pt,scale=0.8]

    \draw (0,-1)  node (m){$1$};
    \draw (2,-1)  node (l1){$2$};
    \draw (2.8,0)  node (l2){$3$};
     \draw (3.2,0)  node (m2){$4$};
    \draw (5,0)  node (l3){$5$};
    \draw (2,1)  node (l4){$$};
  
    \draw (m) --  (i1) ;
    \draw (i1) --  (l1) ;
    \draw (i1) --  (i2) ;
    \draw (i2) --  (l2) ;
  
    \draw (m2) --  (i4) ;
    \draw (i4) --  (l3) ;
    \draw (i4) --  (i3) ;
    \draw (i3) --  (i2) ;
\end{scope}
 \tikzstyle{every node}=[inner sep=2pt,scale=1.5]
   \draw (4,0)  node {$\circ$};
   \draw (9,0)  node {$=$};

\end{tikzpicture}
 \caption{Example of composition of leaf-labeled rooted trees. The labels near the internal nodes 
correspond to the valency of the nodes.}
 \label{fig:treecomposition}
\end{figure}
Let $T_1$ and $T_2$ be leaf-labeled rooted trees  with label sets $A_1, A_2 \subset \PP$ such that 
$|A_1|=|A_2|$. $T_1$ and $T_2$ are said to be \emph{equivalent} and write 
$T_1\sim T_2$ if we
can obtain $T_2$ from $T_1$ by replacing the labels in $T_1$ according to the unique order
preserving bijection between $A_1$ and $A_2$. 

If $\A$ is a set of leaf-labeled trees, $\A$ is said to have the \emph{label substitution
property} if whenever $T_1\sim T_2$ then $T_1 \in \A$ if and only if $T_2 \in \A$. 
$\A$  is said to have the \emph{unique decomposition property} if for every $T \in \A$ then $T \neq 
T_1 \circ T_2
\circ
\cdots \circ T_k $ for trees $T_j \in \A$, i.e., $T$ cannot be written as a nontrivial 
composition of other trees in $\A$. A set $\A$ with these two properties is 
called an
\emph{alphabet} and any tree in $\A$ is called a \emph{letter}. We can also consider alphabets
$\A_S$ that are formed by colored letters, that is, pairs $(T,s)$ where $T\in \A$ and $s \in S$ for 
some
set $S$. A \emph{link} is the composition of two (colored) letters when
defined.

Assume that $\A_S$ is partitioned into equivalence classes of colored letters and let $K$ be the 
set of equivalence classes. For a 
leaf-labeled tree $T$ constructed composing letters from $\A_S$ we denote by $m_j(T)$ the number of 
letters of the equivalence class $j$ that are in $T$. We also denote by $|T|=\sum_{j \in K} 
m_j(T)$ the total number of letters in $T$.
Consider now a partition of the set of links into two parts, that we will call from 
now on \emph{allowed links} $\L(\A_S)$ and 
\emph{forbidden links} $\overline{\L(\A_S)}$. Let $\T_S^n$ and $\overline{\T_S}^n$ for 
$n\ge 1$ be the 
families of trees constructed exclusively with allowed 
links or exclusively with forbidden links respectively, and whose labels are the elements of the 
set $[n]$, each label occuring exactly once. Define $\T_S=\cup_{n\ge 1} \T_S^n$ and 
$\overline{\T_S}=\cup_{n\ge 1} \overline{\T_S}^n$.
In particular, $\T_S^1=\overline{\T_S}^1=\{\bullet_1\}$ is the tree with a single node labeled 
$1$ and we consider letters in $\A_S$ as if they are both in $\T_S$ and $\overline{\T_S}$.

Define the monomials
\begin{align}
\xx^{m(T)}=\prod_{j\in K}x_j^{m_j(T)},
\end{align}
and the generating functions
  \begin{align}  
 F(y)&=\sum_{n \ge 1}\sum_{T \in \T_S^n}\xx^{m(T)}\frac{y^n}{n!},\\
  \overline{F}(y)&=\sum_{n \ge 1}\sum_{T \in 
\overline{\T}_S^n}(-1)^{|T|}\xx^{m(T)}\frac{y^n}{n!},\nonumber
\end{align}
where $y$ and $x_j$ for $j\in K$ are indeterminates.

The following theorem of Drake \cite{Drake2008} reveals a beautiful algebraic relation between 
the exponential generating 
function for the trees constructed using only allowed links and the exponential generating 
function for the trees constructed using 
only forbidden links. Its proof is a consequence of the combinatorial interpretation of the 
composition of exponential generating functions given in \cite{Stanley1999}.

\begin{theorem}[\cite{Drake2008} Theorem 1.3.3]\label{theorem:drake}
We have
\begin{align}
 F^{\left\langle -1 \right\rangle}(y)=\overline{F}(y).
\end{align}
\end{theorem}

There is a gap in the argument in the original proof of Theorem 
\ref{theorem:drake} in \cite{Drake2008}. For the sake of completeness we provide a proof of this 
Theorem fixing this gap.

We begin by considering the set of leaf-labeled rooted trees $T$ constructed as follows: Starting 
from an ordered partition $\pi=(\pi_1,\dots,\pi_{\ell})$ of $[n]$ (that is, where the blocks are 
linearly ordered), $T$ is of the form $T^{a} \circ T_1^{f} 
\circ T_2^{f}\circ \cdots \circ T_{\ell}^{f}$ where $T^{a} \in \T_S$ and $T_i^{f} \in 
\overline{\T_S}$ for all $i$, with the condition that $T_i^{f}$ has label set $\pi_i$ and $T^{a}$ 
has label 
set $\{v(T_i^{f})\}_{i=1}^{\ell}$. Note that different factorizations (different partitions and set 
of subtrees) of the form above can create the same underlying tree $T$. Indeed, the links between 
the tree $T^{a}$ and the trees $T_i^{f}$ can be either in $\L(\A_S)$ or in  
$\overline{\L(\A_S)}$ and hence there could be multiple choices as to where $T^{a}$ finishes and 
$T_i^{f}$ starts. Here we want to consider  
different factorizations of the same tree $T$ as different objects. 
A tree $T$ together with a factorization as above is called a
 \emph{$(\T_S,\overline{\T_S})$-composite tree}. 
 
 \begin{example}\label{example:compositetrees} We can consider the alphabet formed by planar rooted 
trees with exactly one internal node and three leaves with distinct labels. We choose the set of 
forbidden links to be formed by connecting one of the trees in our alphabet as the middle child of 
another tree. Hence the set of allowed links are the ones where trees are not connected in a middle 
child.
  In Figure \ref{fig:examplecompositetree} we illustrate 
two examples of $(\T_S,\overline{\T_S})$-composite trees under these conditions. Note that even 
though both of these composite trees have the same underlying tree they correspond to two 
different factorizations. The underlying ordered partition of the composite tree on the left is 
$1|6\underline{10}\,\underline{11}|9|4|23578$ while the one of the tree on the right is 
$1|6\underline{10}\,\underline{11}|9|4|2|358|7$.
 \end{example}

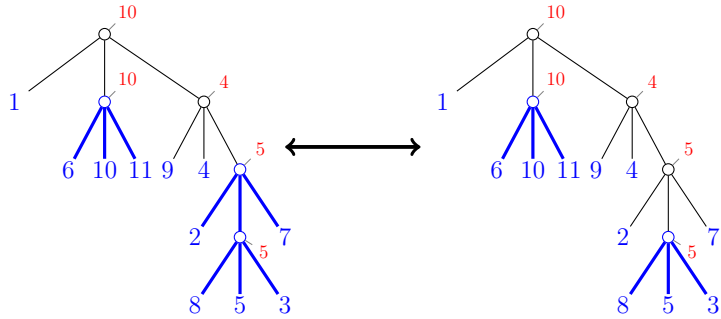
\begin{figure}[ht] 
 \begin{tikzpicture}[scale=0.6]
\draw[line width=0.02in,<->] (5,2) -- (8,2);
\begin{scope}[xshift=0]

\tikzstyle{every node}=[draw,inner sep=0.9mm,scale=0.6]
    \draw [circle] (1,4.5)  node (i1)[pin={[color=red,pin distance=0.2cm]45:$10$}]{$$};
	\draw  [color=blue,circle](1,3)  node (i2)[pin={[color=red,pin distance=0.2cm]45:$10$}]{$$};
   \draw  [color=black,circle](3.2,3)  node (i3)[pin={[color=red,pin distance=0.2cm]45:$4$}]{$$};
\draw  [color=blue,circle](4,1.5)  node (i4)[pin={[color=red,pin distance=0.2cm]45:$5$}]{$$};
   \draw  [color=blue,circle](4,0)  node (i5)[pin={[color=red,pin distance=0.2cm]-20:$5$}]{$$};

\tikzstyle{every node}=[inner sep=1pt, minimum width=14pt,scale=0.8]

    	\draw [color=blue](-1,3)  node (l1){$1$};
    	\draw  [color=blue](0.2,1.5)  node (l2){$6$};
   	\draw  [color=blue](1,1.5)  node (l3){$10$};
	\draw  [color=blue](1.8,1.5)  node (l4){$11$};
   	\draw  [color=blue](2.4,1.5)  node (l5){$9$};
	\draw  [color=blue](3.2,1.5)  node (l6){$4$};
   	\draw  [color=blue](3,0)  node (l7){$2$};
\draw  [color=blue](3,-1.5)  node (l8){$8$};
   	\draw  [color=blue](4,-1.5)  node (l9){$5$};
	\draw  [color=blue](5,-1.5)  node (l10){$3$};
   	\draw  [color=blue](5,0)  node (l11){$7$};

    \draw (i1) --  (l1) ;
    \draw (i1) --  (i2) ;
    \draw (i1) --  (i3) ;
	\draw[color=blue,very thick] (i2) --  (l2) ;
    \draw[color=blue,very thick] (i2) --  (l3) ;
    \draw[color=blue,very thick] (i2) --  (l4) ;
	\draw (i3) --  (l5) ;
    \draw (i3) --  (l6) ;
    \draw(i3) --  (i4) ;
	\draw[color=blue,very thick]  (i4) --  (l7) ;
    \draw[color=blue,very thick]  (i4) --  (i5) ;
    \draw[color=blue,very thick] (i4) --  (l11) ;
   \draw[color=blue,very thick]  (i5) --  (l8) ;
    \draw[color=blue,very thick]  (i5) --  (l9) ;
    \draw[color=blue,very thick] (i5) --  (l10) ;
\end{scope}

\begin{scope}[xshift=270]

\tikzstyle{every node}=[draw,inner sep=0.9mm,scale=0.6]
    \draw [circle] (1,4.5)  node (i1)[pin={[color=red,pin distance=0.2cm]45:$10$}]{$$};
	\draw  [color=blue,circle](1,3)  node (i2)[pin={[color=red,pin distance=0.2cm]45:$10$}]{$$};
   \draw  [color=black,circle](3.2,3)  node (i3)[pin={[color=red,pin distance=0.2cm]45:$4$}]{$$};
\draw  [color=black,circle](4,1.5)  node (i4)[pin={[color=red,pin distance=0.2cm]45:$5$}]{$$};
   \draw  [color=blue,circle](4,0)  node (i5)[pin={[color=red,pin distance=0.2cm]-20:$5$}]{$$};

\tikzstyle{every node}=[inner sep=1pt, minimum width=14pt,scale=0.8]

    	\draw [color=blue](-1,3)  node (l1){$1$};
    	\draw  [color=blue](0.2,1.5)  node (l2){$6$};
   	\draw  [color=blue](1,1.5)  node (l3){$10$};
	\draw  [color=blue](1.8,1.5)  node (l4){$11$};
   	\draw  [color=blue](2.4,1.5)  node (l5){$9$};
	\draw  [color=blue](3.2,1.5)  node (l6){$4$};
   	\draw  [color=blue](3,0)  node (l7){$2$};
\draw  [color=blue](3,-1.5)  node (l8){$8$};
   	\draw  [color=blue](4,-1.5)  node (l9){$5$};
	\draw  [color=blue](5,-1.5)  node (l10){$3$};
   	\draw  [color=blue](5,0)  node (l11){$7$};

    \draw (i1) --  (l1) ;
    \draw (i1) --  (i2) ;
    \draw (i1) --  (i3) ;
	\draw[color=blue,very thick] (i2) --  (l2) ;
    \draw[color=blue,very thick] (i2) --  (l3) ;
    \draw[color=blue,very thick] (i2) --  (l4) ;
	\draw (i3) --  (l5) ;
    \draw (i3) --  (l6) ;
    \draw(i3) --  (i4) ;
	\draw  (i4) --  (l7) ;
    \draw  (i4) --  (i5) ;
    \draw (i4) --  (l11) ;
   \draw[color=blue,very thick]  (i5) --  (l8) ;
    \draw[color=blue,very thick]  (i5) --  (l9) ;
    \draw[color=blue,very thick] (i5) --  (l10) ;
\end{scope}

\end{tikzpicture}
 \caption{Two $(\T_S,\overline{\T_S})$-composite trees.}
 \label{fig:examplecompositetree}
\end{figure}

\begin{lemma}[\cite{Drake2008} Lemma 1.3.2]\label{lemma:132}
The composition $F(\overline{F}(y))$ is the exponential generating function for 
$(\T_S,\overline{\T_S})$-composite trees $T$ weighted by $(-1)^{m_f}\xx^{m(T)}$ where $m_f$ is the 
number of 
letters in the forbidden trees.
\end{lemma}
\begin{proof}
This follows from a classical combinatorial interpretation of composition of exponential generating 
functions (see \cite[Theorem 5.1.4]{Stanley1999}).
\end{proof}

\begin{proof}[Proof of Theorem \ref{theorem:drake}] Using Lemma \ref{lemma:132} we only need to 
show 
that the weighted exponential 
generating function for $(\T_S,\overline{\T_S})$-composite trees is equal to $y$.
 We define a sign-reversing involution $\iota$ on the set of 
$(\T_S,\overline{\T_S})$-composite trees where the only fixed point is the tree with a 
single node (whose factorization is unique).
Let $\mathfrak{T}=T^{a} \circ T_1^{f} \circ T_2^{f}\circ \cdots \circ T_{\ell}^{f}$ be a  
$(\T_S,\overline{\T_S})$-composite tree. 
 Now, recursively and starting at the letter $R$
that contains the root of $\mathfrak{T}$, we move to the child of $R$ with smallest valency among 
the ones that have an allowed link with $R$. Note that the choice of this child is unique since 
the children of $R$ have disjoint sets of descendants and each leaf of $\mathfrak{T}$ has 
a unique label. The fact that the valency is a recursive assignment implies that every child of $R$ 
has a different valency label. We continue recursively repeating the same process until we find a 
letter $R_0$ whose children are either leaves or form forbidden links with $R_0$. The map 
works as follows:
If $R_0$ is a letter of $T^{a}$ then $\iota(\mathfrak{T})$ is the factorization
where the tree starting at $R_0$ with all its descendants is part of the forest of trees with 
forbidden 
links. 
If $R_0$ is not a letter of $T^{a}$ then $\iota(\mathfrak{T})$ is the factorization where 
$T^{a}\circ R_0$ is the new tree with allowed links. Note that the map is well-defined since every 
child of $R_0$ is either a leaf or forms a forbidden link with $R_0$. Note also that the choice of 
$R_0$ does not depend on the factorization of the underlying tree $T$ and every time we apply the 
process the role of $R_0$ changes between being part of the tree with allowed links or being part 
of 
a tree in the forest of trees with forbidden links. Therefore this process is an involution that 
reverses the sign as defined in Lemma \ref{lemma:132}. 
\end{proof}

\begin{example}
In Figure \ref{fig:examplecompositetree} we illustrate two trees (under the same conditions as in 
Example \ref{example:compositetrees}) that are related by the involution $\iota$ described in the 
proof of Theorem \ref{theorem:drake}. In this example $R_0$ is the tree whose 
children have labels $2$, $5$ and $7$. Note that the number of letters in the set of forbidden 
trees changes by one and hence also the weight  $(-1)^{m_f}\xx^{m(T)}$  alternates sign 
between these two composite trees.
\end{example}

\begin{remark}\label{remark:correction}
The proof of Theorem 1.3.3 in \cite{Drake2008} defines the map $\iota$ as follows: 
First select a letter $R_0$ of $T^{a}$ traveling always from the root to the smallest label 
 that has a letter of $T^{a}$ substituted in (it is assumed that the trees are planar and the 
valency rule chooses the leftmost label in the 
subtree). For this letter $R_0$ either all children are 
leaves or it has at least one child forming a link with 
a tree in $\mathfrak{T}$. If every child is a 
leaf then $\iota(\mathfrak{T})$ is the factorization where $R_0$ is considered as a tree in 
$\overline{\T_S}$. Otherwise let $R_1$ be the letter substituted into the child with smallest 
label. If $R_0 \circ R_1$ is an allowed link then $\iota(\mathfrak{T})$ is the factorization where 
$R_0 \circ R_1$ is part of the tree in $\T_S$. Otherwise make $R_0$ part of the forest of 
trees with forbidden links.

The issue here is that $\iota$ is not a well-defined map. For example, 
assume that we are considering 
planar letters with the conditions in Example \ref{example:compositetrees}. Then the 
$(\T_S,\overline{\T_S})$-composite tree in the left of Figure 
\ref{fig:counterexample} does not have a well defined image. In the figure, the 
factorization suggested for the tree in the left has underlying ordered partition 
$\pi=1|234|5|6|7$. The reader can check that the process described in the previous paragraph, will 
assign the new factorization of the tree in the right of Figure \ref{fig:counterexample} that is 
not a $(\T_S,\overline{\T_S})$-composite tree (recall that forbidden trees can only have 
links involving middle children).  

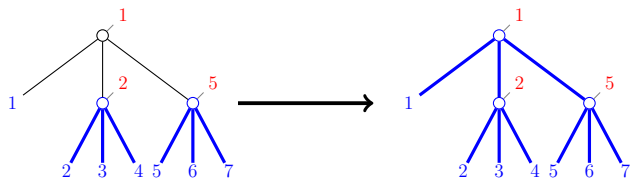
\begin{figure}[ht]
 \begin{tikzpicture}[scale=0.6]
\draw[line width=0.02in,->] (4,0) -- (7,0);
\begin{scope}[xshift=0]
\tikzstyle{every node}=[draw,inner sep=0.9mm,scale=0.6]
    \draw [circle] (1,1.5)  node (i1)[pin={[color=red,pin distance=0.2cm]45:$1$}]{$$};
	\draw  [color=blue,circle](1,0)  node (i2)[pin={[color=red,pin distance=0.2cm]45:$2$}]{$$};
   \draw  [color=blue,circle](3,0)  node (i3)[pin={[color=red,pin distance=0.2cm]45:$5$}]{$$};

\tikzstyle{every node}=[inner sep=1pt, minimum width=14pt,scale=0.6]

    	\draw [color=blue](-1,0)  node (l1){$1$};
    	\draw  [color=blue](0.2,-1.5)  node (l2){$2$};
   	\draw  [color=blue](1,-1.5)  node (l3){$3$};
	\draw  [color=blue](1.8,-1.5)  node (l4){$4$};
   	\draw  [color=blue](2.2,-1.5)  node (l5){$5$};
	\draw  [color=blue](3,-1.5)  node (l6){$6$};
   	\draw  [color=blue](3.8,-1.5)  node (l7){$7$};

    \draw (i1) --  (l1) ;
    \draw (i1) --  (i2) ;
    \draw (i1) --  (i3) ;
	\draw[color=blue,very thick] (i2) --  (l2) ;
    \draw[color=blue,very thick] (i2) --  (l3) ;
    \draw[color=blue,very thick] (i2) --  (l4) ;
	\draw[color=blue,very thick] (i3) --  (l5) ;
    \draw[color=blue,very thick] (i3) --  (l6) ;
    \draw[color=blue,very thick] (i3) --  (l7) ;
\end{scope}

\begin{scope}[xshift=250]
\tikzstyle{every node}=[draw,inner sep=0.9mm,scale=0.6]
    \draw [color=blue,circle] (1,1.5)  node (i1)[pin={[color=red,pin distance=0.2cm]45:$1$}]{$$};
	\draw  [color=blue,circle](1,0)  node (i2)[pin={[color=red,pin distance=0.2cm]45:$2$}]{$$};
   \draw  [color=blue,circle](3,0)  node (i3)[pin={[color=red,pin distance=0.2cm]45:$5$}]{$$};

\tikzstyle{every node}=[inner sep=1pt, minimum width=14pt,scale=0.6]

    	\draw [color=blue](-1,0)  node (l1){$1$};
    	\draw  [color=blue](0.2,-1.5)  node (l2){$2$};
   	\draw  [color=blue](1,-1.5)  node (l3){$3$};
	\draw  [color=blue](1.8,-1.5)  node (l4){$4$};
   	\draw  [color=blue](2.2,-1.5)  node (l5){$5$};
	\draw  [color=blue](3,-1.5)  node (l6){$6$};
   	\draw  [color=blue](3.8,-1.5)  node (l7){$7$};

    \draw[color=blue,very thick]  (i1) --  (l1) ;
    \draw[color=blue,very thick]  (i1) --  (i2) ;
    \draw[color=blue,very thick]  (i1) --  (i3) ;
	\draw[color=blue,very thick] (i2) --  (l2) ;
    \draw[color=blue,very thick] (i2) --  (l3) ;
    \draw[color=blue,very thick] (i2) --  (l4) ;
	\draw[color=blue,very thick] (i3) --  (l5) ;
    \draw[color=blue,very thick] (i3) --  (l6) ;
    \draw[color=blue,very thick] (i3) --  (l7) ;
\end{scope}
\end{tikzpicture}

 \caption{Counterexample to the map defined in \cite{Drake2008}.}
  \label{fig:counterexample}
\end{figure}

\end{remark}

In the following we will be using an alphabet of colored planar binary 
letters of the form:
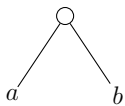
\begin{center}
   \begin{tikzpicture}[scale=0.7]
\begin{scope}[xshift=0,yshift=-1cm]

\tikzstyle{every node}=[draw,inner sep=1mm,scale=0.8]
    \draw [circle] (1,1.5)  node (i1){$$};

\tikzstyle{every node}=[inner sep=1pt, minimum width=14pt,scale=0.8]

    \draw (0,0)  node (m){$a$};
    \draw (2,0)  node (l1){$b$};

    \draw (m) --  (i1) ;
    \draw (i1) --  (l1) ;
\end{scope}

\end{tikzpicture}
\end{center}
where $a<b$ and with colors from $\PP$. The possible links are of the form:
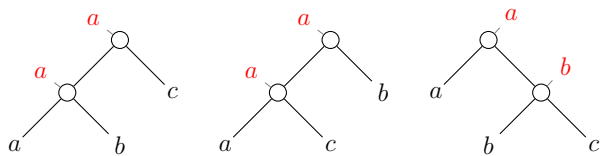
\begin{center}
  \begin{tikzpicture}[scale=0.7]
\begin{scope}[xshift=0,yshift=-1cm]

\tikzstyle{every node}=[draw,inner sep=1mm,scale=0.8]
    \draw [circle] (1,1)  node (i1)[pin={[red,pin distance=0.05in]150:$a$}]{$$};
    \draw [circle] (2,2)  node (i2)[pin={[red,pin distance=0.05in]150:$a$}]{$$};

\tikzstyle{every node}=[inner sep=1pt, minimum width=14pt,scale=0.8]

    \draw (0,0)  node (m){$a$};
    \draw (2,0)  node (l1){$b$};
    \draw (3,1)  node (l2){$c$};

    \draw (m) --  (i1) ;
    \draw (i1) --  (l1) ;
    \draw (i1) --  (i2) ;
    \draw (i2) --  (l2) ;
\end{scope}

\begin{scope}[xshift=4cm,yshift=-1cm]

\tikzstyle{every node}=[draw,inner sep=1mm,scale=0.8]
    \draw [circle] (1,1)  node (i1)[pin={[red,pin distance=0.05in]150:$a$}]{$$};
    \draw [circle] (2,2)  node (i2)[pin={[red,pin distance=0.05in]150:$a$}]{$$};

\tikzstyle{every node}=[inner sep=1pt, minimum width=14pt,scale=0.8]

    \draw (0,0)  node (m){$a$};
    \draw (2,0)  node (l1){$c$};
    \draw (3,1)  node (l2){$b$};

    \draw (m) --  (i1) ;
    \draw (i1) --  (l1) ;
    \draw (i1) --  (i2) ;
    \draw (i2) --  (l2) ;
\end{scope}
\begin{scope}[xshift=7cm,yshift=-1cm]

\tikzstyle{every node}=[draw,inner sep=1mm,scale=0.8]
    \draw [circle] (3,1)  node (i1)[pin={[red,pin distance=0.05in]50:$b$}]{$$};
    \draw [circle] (2,2)  node (i2)[pin={[red,pin distance=0.05in]50:$a$}]{$$};

\tikzstyle{every node}=[inner sep=1pt, minimum width=14pt,scale=0.8]

    \draw (2,0)  node (m){$b$};
    \draw (4,0)  node (l1){$c$};
    \draw (1,1)  node (l2){$a$};

    \draw (m) --  (i1) ;
    \draw (i1) --  (l1) ;
    \draw (i1) --  (i2) ;
    \draw (i2) --  (l2) ;
\end{scope}

\end{tikzpicture}
 \end{center}
 where $a<b<c$.

\subsection{Normalized binary trees}\label{subsection:normalizedtrees}
For each internal node $x$ of a labeled binary tree, let $L(x)$ denote the left child of $x$ and 
$R(x)$
denote its right child.  For each node $x$ of a labeled binary tree $(T,\sigma)$ define
its  \emph{valency} $v(x)$ to be the smallest leaf label of the subtree rooted at $x$. Figure 
\ref{fig:lyndonandvalency} illustrates the valencies of the internal nodes of a labeled 
binary tree. 

We say that a labeled  binary tree  is \emph{normalized} if  the leftmost leaf of each subtree has
the smallest label in the subtree. This is equivalent to requiring that for every internal node $x$,
\begin{align*}
 v(x)=v(L(x)).
\end{align*}

Note that a normalized tree can be regarded as a labeled nonplanar binary tree (or phylogenetic 
tree) that has been drawn in the plane following the convention above. We denote by $\nor_n$ the 
set 
of normalized labeled binary trees on label set $[n]$. 
It is well-known that there are  
$(2n-3)!!:=1\cdot 3 \cdots (2n-3)$  phylogenetic trees 
on $[n]$ and so 
$|\nor_n|=(2n-3)!!$.

 A \emph{Lyndon tree} is a normalized tree $(T,\sigma)$ such that for every internal node $x$ of 
$T$ we have \begin{align} \label{equation:lynnode} v(R(L(x))> v(R(x)).\end{align}

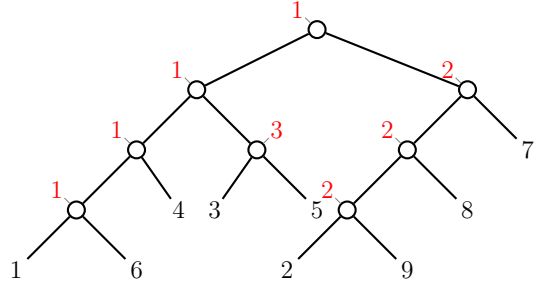
\begin{figure}[ht]
        \centering
        \begin{tikzpicture}[thick,scale=0.8]

\tikzstyle{every node}=[circle,draw,inner sep=0pt, minimum width=8pt,scale=0.8]

    \draw [circle,color=black] (6,4)  node (i1)[pin={[pin distance=0.05in] above left:\color{red}1}]{};
    \draw [circle,color=black] (8.5,3)  node (i2)[pin={[pin distance=0.05in]above left:\color{red}2}]{};
    \draw [circle,color=black] (7.5,2)  node (i3)[pin={[pin distance=0.05in]above left:\color{red}2}]{};
    \draw [color=black] (6.5,1)  node (i4)[pin={[pin distance=0.05in]above left:\color{red}2}]{};
    \draw [color=black] (4,3)  node (i5)[pin={[pin distance=0.05in]above left:\color{red}1}]{};
    \draw [color=black] (5,2)  node (i6)[pin={[pin distance=0.05in]above right:\color{red}3}]{};
    \draw [circle,color=black] (3,2)  node (i7)[pin={[pin distance=0.05in]above left:\color{red}1}]{};
    \draw [color=black] (2,1)  node (i8)[pin={[pin distance=0.05in]above left:\color{red}1}]{};
\tikzstyle{every node}=[inner sep=1pt, minimum width=14pt,scale=0.8]

    \draw (4.3,1)  node (l1){3};
    \draw (5.5,0)  node (l2){2};
    \draw (1,0)  node (l3){1};
    \draw (3,0)  node (l4){6};
    \draw (6,1)  node (l5){5};
    \draw (3.7,1)  node (l6){4};
    \draw (7.5,0)  node (l7){9};
    \draw (9.5,2)  node (l8){7};
    \draw (8.5,1)  node (l9){8};

    \draw (i1) --  (i2) ;
    \draw (i1) --  (i5) ;
    \draw (i2) --  (i3) ;
    \draw (i2) --  (l8) ;
    \draw (i3) --  (i4) ;
    \draw (i3) --  (l9) ;
    \draw (i4) --  (l7) ;
    \draw (i4) --  (l2) ;
    \draw (i5) --  (i6) ;
    
    \draw (i5) --  (i7) ;
    \draw (i6) --  (l5) ;
    \draw (i6) --  (l1) ;
    \draw (i7) --  (l6) ;
    \draw (i7) --  (i8) ;
    \draw (i8) --  (l3) ;
    \draw (i8) --  (l4) ;
\end{tikzpicture}
 \caption{Example of a Lyndon tree. The numbers above the  lines correspond to the valencies of the
internal nodes.}
\label{fig:lyndonandvalency}
  \end{figure}

 We will say that an internal node $x$ of a labeled binary tree $(T,\sigma)$ is a \emph{Lyndon
node} if  (\ref{equation:lynnode}) holds.  
Hence $(T,\sigma)$ is a Lyndon tree if and only if it
is normalized and all its internal nodes are Lyndon nodes. A Lyndon tree is illustrated in Figure 
\ref{fig:lyndonandvalency}. It is known that the set of Lyndon trees with $n$ leaves gives a basis 
for the multilinear component of the free Lie algebra on $n$ generators (see for example 
\cite{Wachs1998}).

\subsection{Colored normalized trees}\label{subsection:colorednormalizedtrees}

We will also be considering labeled binary trees with colored internal nodes. A \emph{colored 
labeled binary tree} is a labeled binary tree such that every internal node $x$ has been assigned a 
\emph{color} $\clr(x) \in \PP$. For a weak composition $\mu \in \wcomp_{n-1}$ we denote $\BT_{\mu}$ 
the set of colored labeled binary trees that contain exactly $\mu_j$ internal nodes colored $j$ for 
each $j$.

A \emph{colored Lyndon tree} is a normalized binary tree such that for any node $x$ that is not
a Lyndon node the following condition must be satisfied:
\begin{align}\label{equation:lyndoncondition}
 \clr(L(x))>\clr(x).
\end{align}
For $\mu \in \wcomp_{n-1}$, let $\lyn_{\mu}$ be the set of colored Lyndon trees in $\BT_{\mu}$ and
$\lyn_{n}=\cup_{\mu \in \wcomp_{n-1}} \lyn_{\mu}$. 
Note that equation (\ref{equation:lyndoncondition}) implies that the monochromatic
Lyndon trees are just the classical Lyndon trees. Figure \ref{fig:coloredlyndontype} shows an 
example of a colored Lyndon tree.

A \emph{colored comb} is a  normalized colored binary tree that satisfies the following
coloring restriction:  for each internal node $x$ whose right child $R(x)$ is not a leaf, 
\begin{align}\label{equation:combcondition}
 \clr(x)>\clr(R(x)).
\end{align}

Let $\comb_{\mu}$ be the set of colored combs in $\BT_{\mu}$ and $\comb_{n}$
the set of all colored combs. Figure \ref{fig:combtype} shows an 
example of a colored comb. Note that in a monochromatic comb every right child has to be 
a leaf and hence they are the classical left combs that are known to give a basis for the 
multilinear component of the free Lie algebra on $n$ generators $\lie(n)$ (see 
\cite[Proposition~2.3]{Wachs1998}). The $\mu$-colored Lyndon trees and combs generalize the 
classical Lyndon trees and combs and both give bases for the $\sym_n$-module $\lie(\mu)$ in 
\cite{Dleon2013} (see also \cite{DleonWachs2013}).

Using Drake's approach we have another perspective to define these types of
trees.

Consider the alphabet $\A_{\PP}$ with letters:
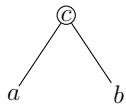
\begin{center}
  \begin{tikzpicture}[scale=0.7,pin distance=0.2cm]
\begin{scope}[xshift=0,yshift=-1cm]

\tikzstyle{every node}=[draw,circle,inner sep=0.5,scale=0.8]
    \draw [circle] (1,1.5)  node (i1)[]{$c$};

\tikzstyle{every node}=[inner sep=1pt, minimum width=14pt,scale=0.8]

    \draw (0,0)  node (m){$a$};
    \draw (2,0)  node (l1){$b$};

    \draw (m) --  (i1) ;
    \draw (i1) --  (l1) ;
\end{scope}

\end{tikzpicture}
\end{center}
where $c \in \PP$ is any color and $a<b$ .

To define the colored Lyndon trees we consider the following forbidden links:
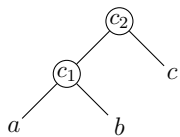
\begin{center}
  \begin{tikzpicture}[scale=0.7,pin distance=0.2cm]
\begin{scope}[xshift=0,yshift=-1cm]

\tikzstyle{every node}=[draw,inner sep=0.5,scale=0.8]
    \draw [circle] (1,1)  node (i1){\small$c_1$};
    \draw [circle] (2,2)  node (i2){\small$c_2$};

\tikzstyle{every node}=[inner sep=1pt, minimum width=14pt,scale=0.8]

    \draw (0,0)  node (m){$a$};
    \draw (2,0)  node (l1){$b$};
    \draw (3,1)  node (l2){$c$};

    \draw (m) --  (i1) ;
    \draw (i1) --  (l1) ;
    \draw (i1) --  (i2) ;
    \draw (i2) --  (l2) ;
\end{scope}

\end{tikzpicture}
\end{center}
with $a<b<c$ and $c_1 \le c_2$, i.e., the colors weakly increase towards the root. Then the
allowed trees are colored Lyndon trees since they satisfy condition 
(\ref{equation:lyndoncondition}) and the forbidden trees are of the form:
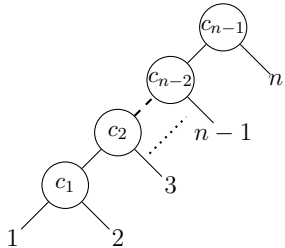
\begin{center}
  \begin{tikzpicture}[scale=0.7]

\tikzstyle{every node}=[draw,circle,inner sep=0, minimum width=22,scale=0.8]

    \draw (1,1)  node (i1){\small $c_{1}$};
    \draw (2,2)  node (i2){\small $c_2$};
    \draw (4,4)  node (i4){\small $c_{n-1}$};
    \draw (3,3)  node (i3){\small $c_{n-2}$};

\tikzstyle{every node}=[inner sep=0pt, minimum width=14pt,scale=0.8]

    \draw (0,0)  node (m){$1$};
    \draw (2,0)  node (l1){$2$};
    \draw (3,1)  node (l2){$3$};
    \draw (4,2)  node (l3){$n-1$};
    \draw (5,3)  node (l4){$n$};

    \draw (m) --  (i1) ;
    \draw (i1) --  (l1) ;
    \draw (i2) --  (l2) ;
    \draw (i3) --  (l3) ;
    \draw (i4) --  (l4) ;
    \draw (i1) --  (i2) ;
    \draw [dashed, thick] (i2) --  (i3) ;
    \draw [dotted, thick] (2.6,1.6) --  (3.3,2.3) ;

    \draw (i3) --  (i4) ;

\end{tikzpicture}
\end{center}
with $c_1\le c_2 \le \cdots \le c_{n-1}$. Since we can completely characterize any such tree by 
defining how many times the color $i$ appears among the $n-1$ internal nodes for each $i \in \PP$, 
and 
considering that \begin{align*}
 h_{n}(\xx):=\sum_{1\le i_1 \le i_2 \le \cdots \le i_n}x_{i_1}x_{i_2}\cdots x_{i_n},
\end{align*}
we obtain 
the following expression for the exponential generating series $\overline{F_{\lyn}}(y)$ for the 
forbidden trees.

\begin{lemma}\label{lemma:forbiddenlyndon}
We have 
\begin{align*}
  \overline{F_{\lyn}}(y)=\sum_{n\ge1}(-1)^{n-1}h_{n-1}(\xx)\frac{y^n}{n!}.
 \end{align*}
\end{lemma}

\subsubsection{Lyndon type of a normalized tree}\label{section:lyndontype}

With a normalized tree $\Upsilon=(T,\sigma)\in \nor_n$ we can associate a set partition 
$\pi^{\lyn}(\Upsilon)$ of the set of 
internal nodes of $\Upsilon$, defined to be the finest set partition satisfying the condition:
\begin{itemize}
 \item for every internal node $x$ that is not Lyndon, $x$ and $L(x)$ belong to the same block 
of $\pi^{\lyn}(\Upsilon)$.
\end{itemize}
For the tree in Figure \ref{fig:coloredlyndontype}, the shaded rectangles indicate the 
blocks of $\pi^{\lyn}(\Upsilon)$. 

Note that the coloring condition (\ref{equation:lyndoncondition}) implies that in 
a 
colored Lyndon tree $\Upsilon$ there are no repeated colors in each block $B$ 
of the partition 
$\pi^{\lyn}(\Upsilon)$ 
associated with $\Upsilon$. Hence after choosing a 
set of $|B|$ colors for the internal nodes in $B$ there is 
a unique way 
to assign the different colors such
that the colored tree $\Upsilon$ is a colored Lyndon tree (the colors must decrease towards 
the 
root in each block of $\pi^{\lyn}(\Upsilon)$). 

Define the \emph{Lyndon type}  $\lyndonlambda(\Upsilon)$ of a normalized tree 
(colored or 
uncolored) $\Upsilon$ to be the (integer) partition whose parts are the block sizes of 
$\pi^{\lyn}(\Upsilon)$. For the tree $\Upsilon$ in Figure 
\ref{fig:coloredlyndontype}, we have $\lyndonlambda(\Upsilon)=(3,2,2,1)$. 

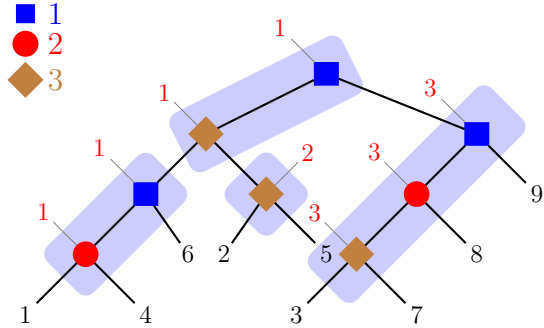
\begin{figure}[ht]
        \centering
        \usetikzlibrary{shapes,snakes}

\begin{tikzpicture}[thick,scale=0.8]

\begin{scope}[xshift=0cm,yshift=1cm]
 \draw [color=blue] (1.5,4)  node (blue){ $1$};
\draw [circle,color=red] (1.5,3.5)  node (red){ $2$};
 \draw [color=brown] (1.5,2.9)  node (brown){$3$};
\tikzstyle{every node}=[fill, draw,inner sep=4pt, minimum width=1pt,scale=0.8]
    
    \draw [color=blue] (1,4)  node (b){};
\draw [circle,color=red] (1,3.5)  node (r){};
    \draw [diamond,color=brown] (1,2.9)  node (g){};
\end{scope}

\node[fill=blue!20,blue!20,draw,rectangle,rounded corners,rotate=45, minimum width=70pt,minimum 
height=30pt,scale=0.8] at  (2.5,1.5) {};
\node[fill=blue!20,blue!20,draw,rectangle,rounded corners,rotate=26.565, minimum width=90pt,minimum 
height=30pt,scale=0.8] at  (5,3.5) {};
\node[fill=blue!20,blue!20,draw,rectangle,rounded corners,rotate=45, minimum width=120pt,minimum 
height=30pt,scale=0.8] at  (7.5,2) {};
\node[fill=blue!20,blue!20,draw,rectangle,rounded corners,rotate=45, minimum width=30pt,minimum 
height=30pt,scale=0.8] at  (5,2) {};
\tikzstyle{every node}=[fill,draw,inner sep=1pt, minimum width=10pt,scale=0.8]

    \draw [color=blue] (6,4)  node (i1)[pin=above left:\color{red}1]{N};
    \draw [color=blue] (8.5,3)  node (i2)[pin=above left:\color{red}3]{N};
    \draw [circle,color=red] (7.5,2)  node (i3)[pin=above left:\color{red}3]{n};
    \draw [diamond, color=brown] (6.5,1)  node (i4)[pin=above left:\color{red}3]{n};
    \draw [diamond,color=brown] (4,3)  node (i5)[pin=above left:\color{red}1]{n};
    \draw [diamond,color=brown] (5,2)  node (i6)[pin=above right:\color{red}2]{n};
    \draw [color=blue] (3,2)  node (i7)[pin=above left:\color{red}1]{N};
    \draw [circle,color=red] (2,1)  node (i8)[pin=above left:\color{red}1]{n};
\tikzstyle{every node}=[inner sep=1pt, minimum width=14pt,scale=0.8]

    \draw (4.3,1)  node (l1){2};
    \draw (5.5,0)  node (l2){3};
    \draw (1,0)  node (l3){1};
    \draw (3,0)  node (l4){4};
    \draw (6,1)  node (l5){5};
    \draw (3.7,1)  node (l6){6};
    \draw (7.5,0)  node (l7){7};
    \draw (9.5,2)  node (l8){9};
    \draw (8.5,1)  node (l9){8};

    \draw (i1) --  (i2) ;
    \draw (i1) --  (i5) ;
    \draw (i2) --  (i3) ;
    \draw (i2) --  (l8) ;
    \draw (i3) --  (i4) ;
    \draw (i3) --  (l9) ;
    \draw (i4) --  (l7) ;
    \draw (i4) --  (l2) ;
    \draw (i5) --  (i6) ;
    
    \draw (i5) --  (i7) ;
    \draw (i6) --  (l5) ;
    \draw (i6) --  (l1) ;
    \draw (i7) --  (l6) ;
    \draw (i7) --  (i8) ;
    \draw (i8) --  (l3) ;
    \draw (i8) --  (l4) ;
\end{tikzpicture}
 \caption{Example of a colored Lyndon tree of type (3,2,2,1). The numbers above the  lines 
correspond to the valencies of the internal nodes.}
\label{fig:coloredlyndontype}
  \end{figure}

\begin{proposition}\label{proposition:generatinglyndontype}We have
 \begin{align*}
  F_{\lyn}(y)= \sum_{n\ge 1} \sum_{\Upsilon \in \nor_n} 
e_{\lyndonlambda(\Upsilon)}(\xx)\dfrac{y^n}{n!}.
 \end{align*}
 \end{proposition}

\begin{proof}
For a colored labeled binary tree $\Psi$ we define its \emph{content} $\mu(\Psi)$ to be the 
weak 
composition $\mu$ where $\mu(i)$ is the number of internal nodes of $\Psi$ that have color $i$.
Let  $\widetilde \Psi$ denote the underlying uncolored labeled binary tree of $\Psi$.
Note that the comments above, together with the fact that
\begin{align*}
 e_{n}(\xx):=\sum_{1\le i_1<i_2<\cdots<i_n}x_{i_1}x_{i_2}\cdots x_{i_n},
\end{align*}
imply that for $\Upsilon \in \nor_n$, the 
generating 
function of
colored Lyndon trees associated with $\Upsilon$ is
\begin{align}\label{equation:contribution}
\sum_{\substack{\Psi \in \lyn_n\\ \widetilde \Psi = \Upsilon}}\xx^{\mu(\Psi)}= 
e_{\lyndonlambda(\Upsilon)}(\xx).
\end{align}
Indeed, the internal 
nodes in a block of size $i$ in the partition  $\pi^{\lyn}(\Upsilon)$ can be colored 
uniquely with any set of $i$ different colors and so the contribution from this block of   
$\pi^{\lyn}(\Upsilon)$  to 
the generating function in (\ref{equation:contribution}) is $e_i(\xx)$. Then

\begin{align*}
\sum_{\Psi \in \lyn_{n}}\xx^{\mu(\Psi)} &=\sum_{\Upsilon \in \nor_{n}}\sum_{\substack{\Psi 
\in \lyn_n\\ \widetilde \Psi = 
\Upsilon}}\xx^{\mu(\Psi)}\\
         &=\sum_{\Upsilon \in \nor_{n}}e_{\lyndonlambda(\Upsilon)}(\xx),
\end{align*}
with the last equality following from (\ref{equation:contribution}).
\end{proof}

We obtain the following theorem as a corollary of Theorem \ref{theorem:drake}.
\begin{theorem}[{\cite[Theorems 1.5 
and 4.3]{Dleon2013}}]\label{theorem:exponentialcompositionallyndontrees}
 We have
 \begin{align}
  \left(\sum_{n\ge1}\sum_{\Upsilon \in 
\nor_n}e_{\lyndonlambda(\Upsilon)}(\xx)\dfrac{y^n}{n!}\right )
^{\left\langle -1 \right\rangle}= \sum_{n\ge1}(-1)^{n-1}h_{n-1}(\xx)\dfrac{y^n}{n!}.
 \end{align}
\end{theorem}

To define the colored combs now we consider forbidden links of the form:
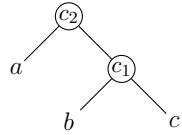
\begin{center}
  \begin{tikzpicture}[scale=0.7,pin distance=0.2cm]
\begin{scope}[xshift=0cm,yshift=0cm]

\tikzstyle{every node}=[draw,inner sep=0.5,scale=0.8]
    \draw [circle] (3,1)  node (i1){\small $c_1$};
    \draw [circle] (2,2)  node (i2){\small $c_2$};

\tikzstyle{every node}=[inner sep=1pt, minimum width=14pt,scale=0.8]

    \draw (2,0)  node (m){$b$};
    \draw (4,0)  node (l1){$c$};
    \draw (1,1)  node (l2){$a$};

    \draw (m) --  (i1) ;
    \draw (i1) --  (l1) ;
    \draw (i1) --  (i2) ;
    \draw (i2) --  (l2) ;
\end{scope}

\end{tikzpicture}
\end{center}
with $a<b<c$ and $c_1 \ge c_2$, i.e., the colors weakly increase towards the root. Then the
allowed trees are colored combs and the forbidden trees look like
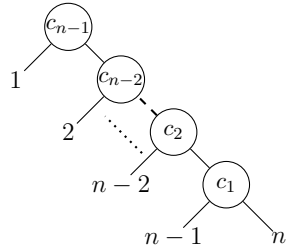
\begin{center}
  \begin{tikzpicture}[scale=0.7]

\tikzstyle{every node}=[draw,inner sep=0,minimum width=22,scale=0.8]

    \draw [circle,color=black] (-1,1)  node (i1){\small$c_1$};
    \draw [circle,color=black] (-2,2)  node (i2){\small $c_2$};
    \draw [circle,color=black] (-4,4)  node (i4){\small $c_{n-1}$};
    \draw [circle,color=black] (-3,3)  node (i3){\small $c_{n-2}$};

\tikzstyle{every node}=[inner sep=0pt, minimum width=14pt,scale=0.8]

    \draw (0,0)  node (m){$n$};
    \draw (-2,0)  node (l1){$n-1$};
    \draw (-3,1)  node (l2){$n-2$};
    \draw (-4,2)  node (l3){$2$};
    \draw (-5,3)  node (l4){$1$};

    \draw (m) --  (i1) ;
    \draw (i1) --  (l1) ;
    \draw (i2) --  (l2) ;
    \draw (i3) --  (l3) ;
    \draw (i4) --  (l4) ;
    \draw (i1) --  (i2) ;
    \draw [dashed, thick] (i2) --  (i3) ;
    \draw [dotted, thick] (-2.6,1.6) --  (-3.3,2.3) ;

    \draw (i3) --  (i4) ;

\end{tikzpicture}
\end{center}
with $c_1\ge c_2 \ge \cdots \ge c_{n-1}$. Then following the same argument as the one before  
Lemma \ref{lemma:forbiddenlyndon} we obtain 
the following expression for the exponential generating series $\overline{F_{\comb}}(y)$ of the 
forbidden trees.

\begin{lemma}\label{lemma:forbiddencomb}
We have
\begin{align*}
  \overline{F_{\comb}}(y)=\sum_{n\ge1}(-1)^{n-1}h_{n-1}(\xx)\frac{y^n}{n!}.
 \end{align*}
\end{lemma}

\subsubsection{Comb type of a normalized tree}\label{section:combtype}
We can associate a new type to  each $\Upsilon \in \nor_n$ in the following 
way: Let $\pi^{\comb}(\Upsilon)$ be the finest set partition of the set of internal nodes of 
$\Upsilon$ satisfying
\begin{itemize}
\item for every pair of internal nodes $x$ and $y$ such that $y$ is a right child of $x$, 
$x$ 
and $y$ belong to the same block of $\pi^{\comb}(\Upsilon)$.
\end{itemize}
We define the \emph{comb type} $\comblambda(\Upsilon)$ of $\Upsilon$ to be the (integer) 
partition whose parts are the sizes of the blocks of $\pi^{\comb}(\Upsilon)$.

Note that the coloring condition (\ref{equation:combcondition}) is closely related to the 
comb type of a normalized tree. The coloring 
condition 
implies that in a colored comb $\Upsilon$ there are no repeated colors in each block $B$ of 
the partition $\pi^{\comb}(\Upsilon)$ associated to $\Upsilon$.
So after choosing 
$|B|$ 
different colors for the 
internal nodes of $\Upsilon$ in $B$,  there is a unique way to assign the colors such 
that $\Upsilon$ is a colored comb (the colors must decrease towards the 
right in each block of $\pi^{\comb}(\Upsilon)$). 
In Figure \ref{fig:combtype} this relation is illustrated.

In the same manner as for Proposition \ref{proposition:generatinglyndontype} and Theorem 
\ref{theorem:exponentialcompositionallyndontrees} we derive the corresponding results for 
colored combs.

\begin{proposition}\label{proposition:generatingcombtype}We have
 \begin{align*}
  F_{\comb}(y)= \sum_{n\ge 1} \sum_{\Upsilon \in \nor_n} 
e_{\comblambda(\Upsilon)}(\xx)\dfrac{y^n}{n!}.
 \end{align*}
 \end{proposition}
\begin{theorem}\label{theorem:exponentialcompositionalcombs}We have
 \begin{align*}
   \left(\sum_{n\ge1}\sum_{\Upsilon \in 
\nor_n}e_{\comblambda(\Upsilon)}(\xx)\dfrac{y^n}{n!}\right )
^{\left\langle -1 \right\rangle}=\sum_{n\ge1}(-1)^{n-1}h_{n-1}(\xx)\dfrac{y^n}{n!}.
 \end{align*}
 \end{theorem}
 
\begin{figure}[ht]
        \centering
        \usetikzlibrary{shapes,snakes}

\begin{tikzpicture}[thick,scale=0.8]

\begin{scope}[xshift=0cm,yshift=1cm]
 \draw [color=blue] (1.5,4)  node (blue){$1$};
\draw [circle,color=red] (1.5,3.5)  node (red){$2$};
 \draw [color=brown] (1.5,2.9)  node (brown){$3$};
\tikzstyle{every node}=[fill, draw,inner sep=4pt, minimum width=1pt,scale=0.8]
    
   \draw [color=blue] (1,4)  node (b){};
\draw [circle,color=red] (1,3.5)  node (r){};
    \draw [diamond,color=brown] (1,2.9)  node (g){};
\end{scope}

\node[fill=blue!20,blue!20,draw,rectangle,rounded corners,rotate=-45, minimum width=70pt,minimum 
height=30pt,scale=0.8] at  (4.5,2.4) {};
\node[fill=blue!20,blue!20,draw,rectangle,rounded corners,rotate=-22, minimum width=100pt,minimum 
height=30pt,scale=0.8] at  (7.25,3.5) {};
height=30pt,scale=0.8] at  (7.5,2) {};
\node[fill=blue!20,blue!20,draw,rectangle,rounded corners,rotate=45, minimum width=30pt,minimum 
height=30pt,scale=0.8] at  (2,1) {};
\node[fill=blue!20,blue!20,draw,rectangle,rounded corners,rotate=45, minimum width=30pt,minimum 
height=30pt,scale=0.8] at  (3,2) {};
\node[fill=blue!20,blue!20,draw,rectangle,rounded corners,rotate=45, minimum width=30pt,minimum 
height=30pt,scale=0.8] at  (6.5,1) {};
\node[fill=blue!20,blue!20,draw,rectangle,rounded corners,rotate=45, minimum width=30pt,minimum 
height=30pt,scale=0.8] at  (7.5,2) {};
\tikzstyle{every node}=[fill,draw,inner sep=0pt, minimum width=1 pt, scale=0.8]

    \draw [circle,color=red] (6,4)  node (i1){n};
    \draw [color=blue]  (8.5,3)  node (i2){N};
    \draw [circle,color=red]  (7.5,2)  node (i3){n};
    \draw [diamond,color=brown]  (6.5,1)  node (i4){n};
    \draw[diamond, color=brown]  (4,3)  node (i5){n};
    \draw [circle,color=red]  (5,2)  node (i6){n};
    \draw [diamond,color=brown] (3,2)  node (i7){n};
    \draw[color=blue]  (2,1)  node (i8){N};
\tikzstyle{every node}=[inner sep=1pt, minimum width=14pt,scale=0.8]

    \draw (4.3,1)  node (l1){2};
    \draw (5.5,0)  node (l2){3};
    \draw (1,0)  node (l3){1};
    \draw (3,0)  node (l4){4};
    \draw (6,1)  node (l5){5};
    \draw (3.7,1)  node (l6){6};
    \draw (7.5,0)  node (l7){7};
    \draw (9.5,2)  node (l8){9};
    \draw (8.5,1)  node (l9){8};

    \draw (i1) --  (i2) ;
    \draw (i1) --  (i5) ;
    \draw (i2) --  (i3) ;
    \draw (i2) --  (l8) ;
    \draw (i3) --  (i4) ;
    \draw (i3) --  (l9) ;
    \draw (i4) --  (l7) ;
    \draw (i4) --  (l2) ;
    \draw (i5) --  (i6) ;
    
    \draw (i5) --  (i7) ;
    \draw (i6) --  (l5) ;
    \draw (i6) --  (l1) ;
    \draw (i7) --  (l6) ;
    \draw (i7) --  (i8) ;
    \draw (i8) --  (l3) ;
    \draw (i8) --  (l4) ;
\end{tikzpicture}
 \caption{Example of a colored comb of comb type $(2,2,1,1,1,1)$.}
\label{fig:combtype}
  \end{figure}

  \begin{remark}
In \cite{Dleon2013} Theorem \ref{theorem:exponentialcompositionallyndontrees}
is proved using a different technique. The proof 
involves the recursive definition of the M\"obius invariant and an EL-labeling of a poset of 
partitions weighted by weak compositions where the ascent-free (or falling) chains coming from the  
EL-labeling are naturally described by colored Lyndon trees. The proof of Theorem 
\ref{theorem:exponentialcompositionalcombs} in \cite{Dleon2013} is a corollary of Theorems
\ref{theorem:exponentialcompositionallyndontrees} and \ref{theorem:equicardinality}.
\end{remark}

Theorems \ref{theorem:exponentialcompositionallyndontrees} and 
\ref{theorem:exponentialcompositionalcombs} together provide a new proof of the following theorem 
that was proved bijectively in 
\cite{Dleon2013}.

\begin{theorem}{{\cite[Theorem 5.4]{Dleon2013}}}\label{theorem:equicardinality}
For every $\mu \in \wcomp$,
\begin{align*}
 |\lyn_{\mu}|= |\comb_{\mu}|.
\end{align*}
\end{theorem}

\begin{remark}
It is interesting to note that under Drake's interpretation 
Theorem \ref{theorem:equicardinality} becomes somewhat more transparent: The two
sets $\lyn_{n}$ and $\comb_{n}$ are constructed using the same alphabet avoiding two 
different sets of forbidden links that are in bijection with each other. It is also interesting 
that both types of forbidden trees have a kind of  ``shape duality", colored Lyndon trees 
have forbidden trees that look like ``left-combs" and colored combs have forbidden trees that look 
like ``right-combs".
\end{remark}

\section{Stirling permutations}\label{section:stirlingpermutations}
Now we consider permutations of the multiset $\{1,1,2,2,\cdots,n,n\}$ satisfying the condition 
that all the numbers between the two occurrences of any fixed number $m$ are larger than $m$, that 
is, multiset permutations $\theta$ satisfying the following condition:
\begin{align}
 \text{if } i<j<k \text{ and } \theta_i=\theta_k=m \text{ then } \theta_j \ge
m.\label{condition:stirling}
\end{align}

For example the 
permutation $12234431$ satisfies condition \eqref{condition:stirling} but the permutation 
$11322344$ does not ($2<3$ and $2$ is between the two occurrences of $3$). The 
permutations satisfying condition \eqref{condition:stirling} were introduced by Gessel and Stanley 
in \cite{StanleyGessel1978} and are known as \emph{Stirling permutations}. For $n\ge0$ we will 
denote the set of Stirling permutations of $\{1,1,2,2,\cdots,n,n\}$ by $\Q_n$.

For $\theta \in \Q_n$, assuming always that $\theta_0=\theta_{2n+1}=0$, consider the sets
\begin{align}
\mathrm{DES}(\theta)&=\{i\,|\,\theta_{i}>\theta_{i+1}\},\nonumber\\
\mathrm{ASC}(\theta)&=\{i\,|\,\theta_{i}<\theta_{i+1}\} \text{ and}\\
\mathrm{PLA}(\theta)&=\{i\,|\,\theta_{i}=\theta_{i+1}\}.\nonumber
\end{align}

These are respectively the sets of \emph{descents}, \emph{ascents} and
\emph{plateaux} defined in \cite{StanleyGessel1978} and \cite{Bona2009}. Let 
$\des(\theta)=|\mathrm{DES}(\theta)|$, $\asc(\theta)=|\mathrm{ASC}(\theta)|$  and 
$\pla(\theta)=|\mathrm{PLA}(\theta)|$ be
their cardinalities. It is an immediate observation that the statistics $\des$ and $\asc$ are 
equidistributed. For this it is enought to consider the function $\rho: \Q_n\rightarrow \Q_n$ that 
reverses a permutation $\rho(\theta)_i:=\theta_{2n+1-i}$. B\'ona \cite{Bona2009} 
proved that these two
statistics are also equidistributed with $\pla$ by
showing that they satisfy the same recurrence relation. Janson, Kuba and Panholzer 
\cite{JansonKubaPanholzer2011} gave a simple combinatorial proof using a bijection of Gessel 
between Stirling permutations and increasing ternary trees.

The triangle of numbers given by any one of these equidistributed statistics are known as the 
\emph{second-order
Eulerian numbers} (see \cite{graham1994concrete}). This terminology
emphasizes that the Stirling
permutations are the second case ($r=2$) of a more general family $\Q_n(r)$ of permutations of the 
multiset $\{1^r,2^r,\cdots,n^r\}$ satisfying condition (\ref{condition:stirling}). Note 
that $\Q(1)=\sym_n$ and $\Q_n(2)=\Q_n$. These 
more general multiset permutations have been also studied (see
\cite{Park1994-1,Park1994-2,Park1994-3,JansonKubaPanholzer2011,KubaPanholzer2011} and
\cite{HaglundVisontai2012}) with the name of \emph{$r$-Stirling permutations} or
\emph{$r$-multipermutations}. We borrow the terminology in \cite{graham1994concrete} and call in
general any statistic that is equidistributed with the descent statistic in $\Q_n(r)$ an
\emph{$r$th-order Eulerian statistic}.

From a permutation in $\Q_{n-1}$ we
can obtain a permutation in $\Q_n$ by inserting the consecutive labels $nn$ in $2n-1$ possible
positions. We have that $\Q_1=\{11\}$ and so $|\Q_1|=1$. By induction we obtain that 
$|\Q_n|=1\cdot 3 \cdots (2n-1)=(2n-1)!!$. 

\subsection{Type of a Stirling permutation}\label{section:stirlingtype}
In this section we define several types associated to a Stirling permutation. These types 
were introduced in \cite{Dleon2013}. 

A \emph{segment} $u$ of a Stirling permutation 
$\theta=\theta_1\theta_2\cdots\theta_{2n}$ is a 
subword  of $\theta$ of the form $u=\theta_{i}\theta_{i+1}\cdots\theta_{i+\ell}$, i.e., 
all the letters of $u$ are adjacent in $\theta$. A \emph{block} in a 
Stirling permutation $\theta$ is a segment of $\theta$ that starts and ends with the same letter. 
For example, $455774$ is a block of $12245577413366$. We define $B_{\theta}(a)$ to be the block of 
$\theta$ that starts and ends with the letter $a$, and define $\mathring{B_{\theta}}(a)$  to be the 
segment 
obtained from $B_{\theta}(a)$ after removing the two occurrences of the letter $a$. For example, 
$B_{\theta}(1)=1224557741$ in $\theta=12245577413366$ and $\mathring{B_{\theta}}(1)=22455774$. 

We call $(a,b)$ an \emph{ascending adjacent pair} of $\theta \in \Q_n$ if $a<b$ and the blocks 
$B_{\theta}(a)$ and 
$B_{\theta}(b)$ are adjacent in $\theta$, i.e., $\theta=\theta^{\prime} 
B_{\theta}(a)B_{\theta}(b)\theta^{\prime\prime}$. An 
\emph{ascending adjacent sequence} of $\theta$ of length $k$ is a 
subsequence  $a_1<a_2<\cdots<a_k$ such that $(a_j,a_{j+1})$  is an ascending adjacent pair for 
$j=1,\dots,k-1$.
Similarly, we call $(a,b)$ a
\emph{terminally nested pair} if $a<b$ and the block $B_{\theta}(b)$ is the last 
block in $\mathring{B_{\theta}}(a)$, i.e., 
$\mathring{B_{\theta}}(a)=\theta^{\prime} 
B_{\theta}(b)$ for some Stirling permutation $\theta^{\prime}$ on a subset of the letters. A 
\emph{terminally nested sequence} of 
$\theta$ of length $k$ is a 
subsequence  $a_1<a_2<\cdots<a_k$ such that $(a_j,a_{j+1})$  is a terminally nested pair for 
$j=1,\dots,k-1$.
If we apply the map $\rho$ that reverses the permutation to the two definitions above we obtain the 
notions of \emph{descending adjacent} and \emph{initially nested} pairs and sequences.

We can associate a \emph{type} to a Stirling permutation $\theta \in \Q_n$ using the different 
types of sequences defined above. We 
define the \emph{ascending adjacent type} $\aalambda(\theta)$, to be the partition whose parts are 
the 
lengths of maximal ascending adjacent sequences; the \emph{terminally nested 
type} $\tnlambda(\theta)$, to be the partition whose parts are the lengths of maximal terminally 
nested sequences; the \emph{descending adjacent type} $\dalambda(\theta)$, to be the 
partition whose parts are 
the 
lengths of maximal descending adjacent sequences; and the \emph{initially nested 
type} $\inlambda(\theta)$, to be the partition whose parts are the lengths of maximal initially 
nested sequences.

\begin{example}
If $\theta=158851244667729933$, the maximal ascending adjacent sequences
are $129$, $467$, $3$, $5$ and $8$ and  
$\aalambda(\theta)=(3,3,1,1,1)$. The maximal descending adjacent sequences
are $1$, $2$, $93$, $4$, $6$, $7$, $5$ and $8$ and  
$\dalambda(\theta)=(2,1,1,1,1,1,1,1)$. The maximal terminally nested 
sequences are $158$, $27$, $3$, $4$, $6$ and $9$ and
$\tnlambda(\theta)=(3,2,1,1,1,1)$. The maximal initially nested 
sequences are $158$, $24$, $6$, $7$, $9$ and $3$ and
$\inlambda(\theta)=(3,2,1,1,1,1)$. All of them are partitions of $n=9$.
\end{example}

Note that since every descent in $\theta$ occurs at the end of a maximal ascending adjacent 
sequence then $\ell(\aalambda(\theta))=\des(\theta)$ where $\ell(\lambda)$ indicates  the number of 
parts of a partition $\lambda$. So $\aalambda$ is a refinement of the $\des$ statistic. In the same 
manner since each plateau can be considered as occurring either at the end of a maximal terminally
nested sequence or at the end of a maximal initially nested sequence then
$\ell(\tnlambda(\theta))=\ell(\inlambda(\theta))=\pla(\theta)$. By a similar argument we have that 
$\ell(\dalambda(\theta))=\asc(\theta)$. Table \ref{table:stirlingpermutationsn3} gives the 
values of $\des$, $\asc$, $\pla$, $\aalambda$, $\dalambda$, $\tnlambda$ and $\inlambda$ for $n=3$. 
For $n=3$ it happens that $\tnlambda$ and $\inlambda$ are equal but this is not true in general.

\begin{table}[ht]
\definecolor{tcA}{rgb}{0.862745,0.862745,0.862745}
\begin{center}
\begin{tabular}{|c|c|c|c|c|c|c|c|}\hline
\textbf{$\theta$} & \textbf{$\des(\theta)$} & \textbf{$\asc(\theta)$} & \textbf{$\pla(\theta)$} &
\textbf{$\aalambda(\theta)$}&\textbf{$\dalambda(\theta)$}&\textbf{$\tnlambda(\theta)$} &
\textbf{$\inlambda(\theta)$}\\\hline
$112233$ & $1$ & $3$ & $3$ & $(3)$     & $(1,1,1)$ & $(1,1,1)$ & $(1,1,1)$ \\\hline
$113322$ & $2$ & $2$ & $3$ & $(2,1)$   & $(2,1)$   & $(1,1,1)$ & $(1,1,1)$ \\\hline
$221133$ & $2$ & $2$ & $3$ & $(2,1)$   & $(2,1)$   & $(1,1,1)$ & $(1,1,1)$ \\\hline
$223311$ & $2$ & $2$ & $3$ & $(2,1)$   & $(2,1)$   & $(1,1,1)$ & $(1,1,1)$ \\\hline
$331122$ & $2$ & $2$ & $3$ & $(2,1)$   & $(2,1)$   & $(1,1,1)$ & $(1,1,1)$ \\\hline
$122331$ & $2$ & $3$ & $2$ & $(2,1)$   & $(1,1,1)$ & $(2,1)$   & $(2,1)$ \\\hline
$112332$ & $2$ & $3$ & $2$ & $(2,1)$   & $(1,1,1)$ & $(2,1)$   & $(2,1)$ \\\hline
$133122$ & $2$ & $3$ & $2$ & $(2,1)$   & $(1,1,1)$ & $(2,1)$   & $(2,1)$ \\\hline
$122133$ & $2$ & $3$ & $2$ & $(2,1)$   & $(1,1,1)$ & $(2,1)$   & $(2,1)$ \\\hline
$133221$ & $3$ & $2$ & $2$ & $(1,1,1)$ & $(2,1)$   & $(2,1)$   & $(2,1)$ \\\hline
$221331$ & $3$ & $2$ & $2$ & $(1,1,1)$ & $(2,1)$   & $(2,1)$   & $(2,1)$ \\\hline
$233211$ & $3$ & $2$ & $2$ & $(1,1,1)$ & $(2,1)$   & $(2,1)$   & $(2,1)$ \\\hline
$331221$ & $3$ & $2$ & $2$ & $(1,1,1)$ & $(2,1)$   & $(2,1)$   & $(2,1)$ \\\hline
$123321$ & $3$ & $3$ & $1$ & $(1,1,1)$ & $(1,1,1)$ & $(3)$     & $(3)$ \\\hline
$332211$ & $3$ & $1$ & $3$ & $(1,1,1)$ & $(3)$     & $(1,1,1)$ & $(1,1,1)$ \\\hline
\end{tabular}
\end{center}
\caption{Stirling permutations and statistics for $n=3$.}
\label{table:stirlingpermutationsn3}
\end{table}

The following proposition was proved in \cite{Dleon2013}.

\begin{proposition}[{\cite[Proposition 4.6]{Dleon2013}}]\label{proposition:propxi}There is a 
bijection $\xi:\Q_n \rightarrow \Q_n$ that satisfies:	
 \begin{enumerate}
  \item $(i,j)$ is an ascending adjacent pair in $\theta$ if and only if $(i,j)$ is a terminally 
nested pair in $\xi(\theta)$,
\item $\tnlambda(\xi(\theta))=\aalambda(\theta)$.
 \end{enumerate}
\end{proposition}

From Proposition \ref{proposition:propxi} we see that $\aalambda$ and $\tnlambda$ are 
equidistributed on $\Q_n$. Since the map $\rho$ 
that reverses $\theta$ (defined above) also implies $\aalambda \cong \dalambda$ and $\tnlambda 
\cong  \inlambda$ we have as corollaries the following two results.

\begin{theorem}\label{theorem:equilambdas}
The types $\aalambda$, $\dalambda$, $\tnlambda$ and $\inlambda$ are equidistributed.
\end{theorem}

\begin{corollary}[\cite{Bona2009}]
The statistics $\des$, $\asc$ and $\pla$ are equidistributed.
\end{corollary}

\begin{remark}\label{remark:gesselternary}
Theorem \ref{theorem:equilambdas} can be proved in a different way using a bijection of 
Gessel between Stirling permutations and increasing planar ternary trees (see 
\cite{JansonKubaPanholzer2011}). The idea is that from the perspective of increasing planar ternary 
trees (reading the types after the bijection), the equidistributivity of the types 
$\aalambda$, $\dalambda$, $\tnlambda$ and $\inlambda$ is a consequence of the bijection on the 
set of increasing planar ternary trees defined by reordering
simultaneously the $3$ children of every internal node of a ternary tree using a fixed 
permutation. Since there are $6$ permutations in $\sym_3$, proving Theorem 
\ref{theorem:equilambdas} in this way also reveals that there are two other different types 
equidistributed with the ones discussed here. We leave the details of this proof to the reader.
\end{remark}

The link between normalized trees and Stirling permutations is given by the following proposition 
in \cite{Dleon2013}.

\begin{proposition}[{\cite[Proposition 4.8]{Dleon2013}}]\label{proposition:propgamma}There is a 
bijection\footnote{This bijection appeared first in \cite{Dotsenko2012}.} $\gamma: \nor_n 
\rightarrow 
\Q_{n-1}$ that satisfies for each $\Upsilon \in \nor_n$ 
 \begin{enumerate}
\item $\aalambda(\gamma(\Upsilon))=\lyndonlambda(\Upsilon)$,
\item $\tnlambda(\gamma(\Upsilon))=\comblambda(\Upsilon)$.
 \end{enumerate}
\end{proposition}

\begin{proof}[Proof of Theorem \ref{theorem:exponentialcompositional}]
This is a corollary of the results in Propositions 
\ref{proposition:propxi} and 
\ref{proposition:propgamma}; and  Theorems \ref{theorem:exponentialcompositionallyndontrees} and 
\ref{theorem:exponentialcompositionalcombs}.
\end{proof}


\subsection{$r$-Stirling permutations or $r$-multipermutations}\label{section:multipermutations}

An \emph{$r$- Stirling permutation} or \emph{$r$-multipermutation} is a permutation of the multiset
$\{1^r,2^r,\cdots,n^r\}$ satisfying condition (\ref{condition:stirling}). We denote the 
set of $r$-Stirling permutations by $\Q_n(r)$. For example, $1 2 3 3 3 2 2 1 5 5 5 
1 4 6 6 6 4 4$ is 
in $\Q_n(3)$. In particular $\Q_n(2)=\Q_n$ and $\Q_n(1)=\sym_n$. We want to extend the 
right-hand side of equation (\ref{equation:exponentialcompositionalinverse}) to the generality of 
$r$-Stirling permutations.

For a permutation $\theta \in \Q_n(r)$, assuming always that $\theta_0=\theta_{rn+1}=0$, we
define the sets $\mathrm{DES}$ and $\mathrm{ASC}$ and the statistics $\des$ and $\asc$ as before.
Let $n(\theta,i)$ denote the number of occurrences of the label $\theta_i$ in the subword
$\theta_1\theta_2\cdots\theta_{i}$.  We use a refinement of $\mathrm{PLA}$ defined in
\cite{JansonKubaPanholzer2011}, the set
\begin{align}
\mathrm{PLA}_j(\theta)&=\{i\,|\,\theta_{i}=\theta_{i+1},\,n(\theta,i)=j\}\nonumber
\end{align}
of \emph{$j$-plateaux} (plateaux between the occurrences
$j$ and $j+1$ of a label) and the statistic $\pla_j(\theta)=|\mathrm{PLA}_j(\theta)|$ its 
cardinality. In 
\cite{JansonKubaPanholzer2011} it is shown that  $\des$, $\asc$ and $\pla_j$ are 
equidistributed in $\Q_n(r)$.

\subsubsection{Types on the set of $r$-Stirling 
permutations}\label{subsection:typesmultipermutation}
A \emph{block} $B_{\theta}(a)$ in an $r$-Stirling permutation $\theta$ is a segment of $\theta$ 
that starts and ends with $a$ and contains all the occurrences of $a$ in $\theta$. For example, 
$B_{\theta}(1)=122214555441$ is a block of $\theta=122214555441333$. Removing 
all occurrences of $a$ in $B_{\theta}(a)$ gives a sequence $\mathring{B_{\theta}}(a)$  of 
(possibly empty) $r$-Stirling permutations $\mathring{B_{\theta}}(a)_j$ for $j=1,\dots,r-1$. For 
example $\mathring{B_{\theta}}(1)=(222,455544)$.

We call $(a,b)$ an \emph{ascending adjacent pair} in $\theta \in \Q_n(r)$ if $a<b$ and the blocks 
$B_{\theta}(a)$ and 
$B_{\theta}(b)$ are adjacent in $\theta$, i.e., $\theta=\theta^{\prime} 
B_{\theta}(a)B_{\theta}(b)\theta^{\prime\prime}$. An 
\emph{ascending adjacent sequence} of $\theta$ of length $k$ is a 
subsequence  $a_1<a_2<\cdots<a_k$ such that $(a_j,a_{j+1})$  is an ascending adjacent pair for 
$j=1,\dots,k-1$.
Similarly, for $j \in [r-1]$ we call $(a,b)$ a
\emph{$j$-terminally nested pair} if $a<b$ and the block $B_{\theta}(b)$ is the last 
block in $\mathring{B_{\theta}}(a)_j$, i.e., 
$\mathring{B_{\theta}}(a)_j=\theta^{\prime} 
B_{\theta}(b)$ for some $r$-Stirling permutation $\theta^{\prime}$. A 
\emph{$j$-terminally nested sequence} of 
$\theta$ of length $k$ is a 
subsequence  $a_1<a_2<\cdots<a_k$ such that $(a_s,a_{s+1})$  is a $j$-terminally nested pair for 
$s=1,\dots,k-1$.
If we apply the map $\rho$ that reverses the permutation to the two definitions above we obtain the 
notions of \emph{descending adjacent} and \emph{$j$-initially nested} pairs and sequences.

We then associate a \emph{type} to an $r$-Stirling permutation $\theta \in \Q_n(r)$ in different 
ways according to the lengths of maximal sequences of a given type as before. We 
define in this way the \emph{ascending adjacent type} $\aalambda(\theta)$, the 
\emph{$j$-terminally nested 
type} $\tnlambda_j(\theta)$, the \emph{descending adjacent type} $\dalambda(\theta)$ and the  
\emph{$j$-initially nested type} $\inlambda_j(\theta)$. Note that similar to the case $r=2$, in the 
general case 
$\aalambda$ refines $\des$, $\dalambda$ refines $\asc$ and both $\tnlambda_j$ and $\inlambda_j$ 
refine $\pla_j$.

The proof of the following theorem is similar to the proof of Theorem \ref{theorem:equilambdas} in 
\cite{Dleon2013}.

\begin{theorem}\label{theorem:equilambdasgeneral}
The types $\aalambda$, $\dalambda$, $\tnlambda_j$ and $\inlambda_j$ for all $j=1,\dots,r-1$ are 
equidistributed.
\end{theorem}
\begin{remark}

We can also prove Theorem \ref{theorem:equilambdasgeneral} following the same idea discussed in 
Remark \ref{remark:gesselternary}. This time we use instead a more general bijection of Gessel 
between $r$-Stirling permutations and increasing planar $(r+1)$-ary trees.
\end{remark}

With the more general definitions in place we can consider the 
family of symmetric functions
\begin{align}\label{definition:s2}
 \SP_n^{(r)}(\xx)=\sum_{\theta \in \Q_n(r)}e_{\lambda(\theta)}(\xx),
\end{align}
where $\lambda(\theta)$ is any of the types of $\theta$ defined above.

The question is to determine if this more general definition provides interesting results and 
has any combinatorial applications for some $r \neq 2$. We will show in 
Section \ref{section:permutations} that for $r=1$ there is a positive answer with a very similar 
story to the one for $r=2$.

\section{The case $r=1$ of standard permutations of $[n]$}\label{section:permutations}
In this section we prove Theorem \ref{theorem:exponentialmultiplicative}. To prove this theorem we 
first introduce a combinatorial interpretation of the multiplicative inverse of an 
exponential generating function in terms of words with allowed and
forbidden links that follows from a theorem discovered by Fr\"oberg 
\cite{Froberg1975}, Carlitz-Scoville-Vaughan \cite{CarlitzScovilleVaughan1976} and Gessel 
\cite{Gessel1977}.
The theory outlined in \cite{CarlitzScovilleVaughan1976} and \cite{Gessel1977} is more general and 
applies to a larger family of counting algebras (as defined in \cite{Gessel1977}) and not only to 
exponential generating functions. Here we give a simplified description that applies to the 
exponential generating function result.
\subsection{Combinatorial interpretation of the multiplicative 
inverse}\label{subsection:gesselinterpretation}
We consider permutations of any finite label set $A\subset \PP$.
In particular any label by itself and the empty word $\emptyset$ are examples of 
permutations. For any two permutations $w_1$ and $w_2$ that have disjoint label sets $A_1,A_2 
\subset \PP$ we define the 
\emph{product} $w_1w_2$ to be the permutation with label set $A_1 \cup A_2$ constructed by 
concatenation. 
If these conditions are not satisfied the product is not defined. For example if $w_1=1345$ and 
$w_2=276$ then $w_1w_2=1345276$. Note 
that the product is associative and so expressions like $w_1w_2\cdots w_k$
are well defined. Two permutations $w_1$ and $w_2$ with label sets $A_1, A_2 \subset \PP$ such that 
$|A_1|=|A_2|$ are said to be \emph{equivalent}, and we write 
$w_1\sim w_2$, if we can obtain $w_2$ from $w_1$ by replacing the labels in $w_1$ according to the 
unique order preserving bijection between $A_1$ and $A_2$.
If $\A$ is a set of permutations, $\A$ is said to have the \emph{label substitution
property} if whenever $w_1\sim w_2$ then $w_1 \in \A$ if and only if $w_2 \in \A$.
In a set $\A$ of permutations, we call $w \in \A$ \emph{irreducible} 
if $w$ is a nonempty permutation such that $w=w_1w_2$ and $w_1,w_2 \in \A$ imply either $w_1=w$ or 
$w_2=w$, i.e., a permutation  that cannot be decomposed as the concatenation of other permutations 
in $\A$. A set 
$\A$ of permutations is said to have the \emph{unique decomposition property} if all its 
permutations are irreducible.
We call an \emph{alphabet} a set $\A$ of permutations that has the label substitution property and 
the unique decomposition 
property. The elements of $\A$ are called the
\emph{letters} of the alphabet. The set of permutations (including the empty permutation) that can 
be constructed 
as the product of letters in $\A$ is denoted $\A^*$. We can also consider alphabets
$\A_S$ that are formed by colored letters, i.e., pairs $(a,s)$ where $a\in \A$ and $s \in S$ for 
some set $S$. We call a \emph{link} the product of two (colored) letters. 
Assume that the set $\A_S$ of colored letters is partitioned into equivalence classes and let 
$K$ be the set of equivalence classes. For a 
permutation 
$w \in \A_S^*$ we denote by $m_j(w)$ the number of letters from the class $j \in K$ that are 
present 
in $w$ and $|w|:=\sum_{j\in K}m_j(w)$ the \emph{length} of $w$.
Consider a partition of the set of links into two parts that we call the 
\emph{allowed 
links} $\L(\A_S)$ and 
\emph{forbidden links} $\overline{\L(\A_S)}$. Let $\W_S^n$ be the set of 
permutations with underlying label set $[n]$ for $n \ge 0$ constructed with only allowed links and 
let 
$\overline{\W_S}^n$ the ones constructed 
using only forbidden links. Define  $\W_S=\cup_{n\ge 0} \W_S^n$ and 
$\overline{\W_S}=\cup_{n\ge 0} \overline{\W_S}^n$. In particular, 
$\W_S^0=\overline{\W_S}^0=\emptyset$ and we consider letters in $\A_S$ as if they are both in 
$\W_S$ and $\overline{\W_S}$.

Define the monomials
\begin{align*}
X^{m(w)}=\prod_{j \in K} x_j^{m_j(w)},
\end{align*}
and the generating functions
  \begin{align*}  
 F(y)&=\sum_{n \ge 0}\sum_{w \in \W_S^n}X^{m(w)}\frac{y^n}{n!},\\
  \overline{F}(y)&=\sum_{n \ge 0}\sum_{w \in \overline{\W_S}^n}(-1)^{|w|}X^{m(w)}\frac{y^n}{n!}.
\end{align*}

\begin{theorem}[c.f.\cite{Gessel1977}]\label{theorem:gessel}
We have
\[
 F^{-1}(y)=\overline{F}(y).
\]
\end{theorem}

For the sake of completeness we provide a proof of Theorem \ref{theorem:gessel}. The idea of the 
proof is the one that appears in \cite{Gessel1977} where a more general version of this theorem 
is proved.

To prove Theorem \ref{theorem:gessel} we consider the set of permutations of the form $w=w_1w_2$ 
where 
$w_1\in \W_S$ and $w_2\in \overline{\W_S}$. Note that if $w \neq \emptyset$ then $w$ has exactly 
two different factorizations. Indeed, if $w=a_1a_2\dots a_n$ then either there is a number $1\le k< 
n$ such 
that $a_ka_{k+1} \in \L(\A_S)$ but $a_{k+1}a_{k+2} \in \overline{\L(\A_S)}$, or $a_ia_{i+1} \in 
\L(\A_S)$ for all $i$ (in which case we let $k=n$) or $a_ia_{i+1} \in \overline{\L(\A_S)}$ for all 
$i$ (in which case we let $k=0$). Assuming that $a_0=a_{n+1}=\emptyset$ the two valid 
factorizations of $w$ are $(a_0\dots a_{k+1})(a_{k+2}\dots a_{n+1})$ and $(a_0\dots 
a_{k})(a_{k+1}\dots a_{n+1})$. Here we want to consider the two different factorizations of $w$ as 
different objects. We call these objects (a permutation $w$ together with its factorization) 
\emph{$(\W_S,\overline{\W_S})$-composite permutations}. 

\begin{lemma}\label{lemma:compositewords}
The multiplication $F(y)\overline{F}(y)$ is the exponential generating function for 
$(\W_S,\overline{\W_S})$-composite permutations $w$ weighted by $(-1)^{m_f}\xx^{m(w)}$ where $m_f$ 
is the 
number of letters in the forbidden permutation.
\end{lemma}
\begin{proof}
This follows from the combinatorial interpretation of multiplication of exponential generating 
functions in \cite[Proposition 5.1.1]{Stanley1999}.
\end{proof}

\begin{proof}[Proof of Theorem \ref{theorem:gessel}] Using Lemma \ref{lemma:compositewords} we only 
need to show 
that the weighted exponential 
generating function for $(\W_S,\overline{\W_S})$-composite permutations is equal to $1$.
 We define a sign-reversing involution $\iota$ on the set of 
$(\W_S,\overline{\W_S})$-composite permutations where the only fixed point is the empty permutation 
(whose factorization is unique). Let $w$ be a nonempty $(\W_S,\overline{\W_S})$-composite 
permutation, by the comments above we know that the underlying permutation of $w$ can be associated 
with two different factorizations $w$ and $w^{\prime}$. We then define 
$\iota(w)=w^{\prime}$ if $w\neq \emptyset$ and $\iota(\emptyset)=\emptyset$. This process is an 
involution that reverses the sign as defined in Lemma \ref{lemma:compositewords} since $w$ and 
$w^{\prime}$ differ by one in the number of letters in the forbidden permutation.
\end{proof}

\subsection{Colored permutations}
Let $S$ be any subset of $\PP$.
 A \emph{colored permutation} is a permutation $\sigma \in \sym_n$ in which each letter
$j \in [n]$ has been asigned a color $\clr(j) \in S$ with the condition that for
every occurrence of an ascending adjacent pair $\sigma(j)<\sigma(j+1)$ in $\sigma$ it must happen 
that $\clr(\sigma(j))>\clr(\sigma(j+1))$. For example for $S=[3]$,
 $\sigma=1^{3}2^{2}4^{1}3^{3}5^{2}$ is a colored permutation with the 
colored letters 
$(i,\clr(i))$ represented as $i^{\clr(i)}$. Since in any ascending adjacent sequence of 
$\sigma \in \sym_n$ the colors need to strictly decrease, $e_{\aalambda(\sigma)}(\xx)$ enumerates 
the colored permutations with colors in $S=\PP$ and underlying uncolored permutation $\sigma$.

\begin{proof}[Proof of Theorem \ref{theorem:exponentialmultiplicative}]
Let $\A_S$ be the alphabet with colored letters $a^c$ where $a \in [n]$ and $c \in \PP$. Consider 
the set of forbidden links $\overline{\L(\A_s)}$ to be of the form $a^{c_1}b^{c_2}$ with
$a<b$ and $c_1 \le c_2$. The forbidden permutations are
of the form $1^{c_1}2^{c_2}\cdots n^{c_n}$ with $c_1 \le c_2 \le \cdots \le c_n$. Since each 
of the forbidden colored permutations is completely determined after selecting a multiset of colors, 
the
generating polynomial of the forbidden permutations is $h_{n}(\xx)$ and 

\[\overline{F}(y)=\sum_{n\ge0}(-1)^{n}h_{n}(\xx)\frac{y^n}{n!}.\]
 
The allowed permutations are colored permutations whose generating function corresponds by the 
comments 
above to the right-hand 
side of equation (\ref{equation:exponentialmultiplicativeinverse}). Applying Theorem 
\ref{theorem:gessel} 
completes the proof of this theorem.
\end{proof}

\section{Specializations}\label{section:specializations}

A \emph{composition} $\nu$ is a weak composition such that $\nu(i)\neq 0$ for all $i \le 
\ell(\nu)$, where $\ell(\nu)$ is the largest position in $\nu$ with a nonzero entry.
In other words, a composition can be thought of as a finite weak composition with strictly positive 
parts. 
For example $(3,2,1,1,2,1)$ is a composition of $10$. We write $\comp_n$ to denote the set of 
compositions of $n$. 
To every composition $\nu \in \comp_n$ we can associate a partition 
$\lambda(\nu)$ as the nonincreasing rearrangement of the parts of $\nu$. For example, 
$\lambda(3,2,1,1,2,1)=(3,2,2,1,1,1)$.

The following Proposition is a standard result in the theory of symmetric functions and can be 
obtained as a special case of a more general Theorem of E\u gecio\u glu and Remmel 
in \cite{OmerRemmel1991}.

\begin{proposition}[{c.f. \cite[Theorem 2.3]{OmerRemmel1991}}]\label{proposition:htoe} For every 
$n\ge 
0$,
\begin{align*}
 h_n(\xx)=(-1)^n\sum_{\nu \in \comp_n}(-1)^{\ell(\nu)}e_{\nu}(\xx).
\end{align*}
\end{proposition}


Let $E:\Lambda \rightarrow \QQ[t]$ be the map defined by $E(e_i(\xx))=t$ for all $i \ge 1$ and 
$E(1)=1$. Since $E$ is defined on the generators $e_i$ it is immediate to check that 
$E$ is an algebra 
homomorphism or \emph{specialization}.
In the same manner it is also easy to verify that the specialization $E$ extends to a 
specialization $\tilde E:\Lambda[[y]] \rightarrow \QQ[t][[y]]$ in the algebra of power series in 
$y$ with symmetric function coefficients in $\Lambda$ (with variables $\xx$) defined by applying 
$E$ coefficientwise. Moreover,  it is also true and easy 
to verify that $\tilde E$ is a monoid homomorphism  $(\Lambda[[y]],\circ) 
\rightarrow (\QQ[[y]],\circ)$, where $\circ$ indicates composition of power series, i.e., for power 
series $f,g \in \Lambda[[y]]$
$\tilde E(f[g(y)])=\tilde E(f)[\tilde E(g)(y)]$. Note that for any $\lambda\vdash n$ we have that 
\begin{align}\label{equation:specializee}
E(e_{\lambda}(\xx))=t^{\ell(\lambda)}. 
\end{align}

\begin{lemma}For every $n \ge 1$,
\begin{align}\label{equation:specializeh}
 E(h_n(\xx))= t(t-1)^{n-1}.
\end{align}
\end{lemma}
\begin{proof}
Using Proposition \ref{proposition:htoe} and equation (\ref{equation:specializee}),
 \begin{align*}
  E(h_n(\xx))&=E\left((-1)^n\sum_{\nu \in 
\comp_n}(-1)^{\ell(\nu)}e_{\nu}(\xx)\right)\\
&=(-1)^n\sum_{\nu \in \comp_n}(-t)^{\ell(\nu)}\\
&=(-1)^n\sum_{k=1}^{n}\binom{n-1}{k-1}(-t)^{k}\\
&=(-1)^n(-t)(1-t)^{n-1},
 \end{align*}
 where $\binom{n-1}{k-1}$ is the number of compositions of $n$ into $k$ parts.
\end{proof}

Applying the specialization $E$ to definition (\ref{definition:s2}), using equation 
(\ref{equation:specializee}) and applying the observation in Section 
\ref{section:stirlingpermutations} that $\ell(\aalambda(\theta))=\des(\theta)$, we obtain 
\begin{align*}
 E(\SP_n^{(r)}(\xx))&=E\left(\sum_{\theta \in \Q_n(r)}e_{\aalambda(\theta)}(\xx)\right)\\
&=\sum_{\theta \in \Q_n(r)}t^{\ell(\aalambda(\theta))}\\
&=\sum_{\theta \in \Q_n(r)}t^{\des(\theta)}\\
&:=A^{(r)}_{n}(t),
\end{align*}
the \emph{$r$-th order Eulerian polynomial}.

\begin{remark}
Applying the specialization $\tilde E$ to equation 
(\ref{equation:exponentialcompositionalinverse}) and using equations (\ref{equation:specializee}) 
and (\ref{equation:specializeh}) we obtain as a corollary Theorem 
\ref{theorem:specializedexponentialcompositional}. In the same manner, applying $\tilde E$ to 
equation 
(\ref{equation:exponentialmultiplicativeinverse}) we obtain as a corollary Riordan's result 
(Theorem 
\ref{theorem:riordan}).
\end{remark}


\section{Connections and future directions}\label{section:futuredirections}
In this section we discuss some instances where the functions $\SP_n^{(r)}(\xx)$ appear for the 
cases $r=1$ and $r=2$.
\subsection{Multiplicative inverse and Lagrange inversion}

Propositions \ref{proposition:ordinarymultiplicative} and \ref{proposition:ordinarycompositional}, 
and Theorems \ref{theorem:exponentialmultiplicative} and \ref{theorem:exponentialcompositional}  are
closely related to the problem of finding multiplicative and compositional inverses of 
general exponential generating functions.

We denote by $\widetilde{\SP}^{(r)}_n(h_1,h_2,\dots)$ the symmetric function $\SP_n^{(r)}(\xx)$ 
written as a polynomial in the generators $h_i(\xx)$, i.e., $\widetilde{\SP}^{(r)}_n \in 
R[h_1,h_2,\dots]$. Let
\begin{align}
 F(y)=\sum_{n\ge0}f_n\frac{y^n}{n!},
\end{align}
where the $f_n$ are in some commutative ring $A$.

Recall that the multiplicative inverse $F^{-1}$ exists if and only if $f_0$ is a unit in $A$. 
The compositional inverse $F^{\left\langle-1\right\rangle}$ exists if and only if $f_0=0$ and $f_1$ 
is a unit in $A$.
The reader can check that in the former case, when we apply the specialization $h_{n}\mapsto 
f_{n}/f_{0}$ for $n\ge 0$ to equation 
(\ref{equation:exponentialmultiplicativeinverse}) we obtain that the $n$-th coefficient in the 
power series $F^{-1}$ is 
\begin{align}\label{equation:inversionmultiplicative}
 n![x^n]F^{-1}(y)= (-1)^{n}f_0^{-1}\widetilde {\SP}^{(1)}_{n}(f_1/f_0,f_2/f_0,\dots).
\end{align}
In the latter case, after applying the specialization $h_{n}\mapsto f_{n+1}/f_1$ for 
$n\ge 1$ to equation (\ref{equation:exponentialcompositionalinverse}) we obtain that the $n$-th 
coefficient of $F^{\left\langle-1\right\rangle}$ is 
\begin{align}\label{equation:inversioncompositional}
 n![x^n]F^{\left\langle-1\right\rangle}(y)=(-1)^{n-1}f_1^{-n}\widetilde 
{\SP}^{(2)}_{n-1}(f_2/f_1,f_3/f_1,\dots).
\end{align}

Equations \eqref{equation:inversionmultiplicative} and \eqref{equation:inversioncompositional} 
contain another interesting piece of information. 
Let $f_n=\SP_{n}^{(1)}(\xx)$ in \eqref{equation:inversionmultiplicative} and 
$f_{n}=\SP_{n-1}^{(2)}(\xx)$ in \eqref{equation:inversioncompositional} for all $n$. 
Then Theorems \ref{theorem:exponentialmultiplicative} and \ref{theorem:exponentialcompositional} 
imply that the left-hand side of \eqref{equation:inversionmultiplicative} becomes $h_n(\xx)$ and 
the left-hand side of  \eqref{equation:inversioncompositional} becomes $h_{n-1}(\xx)$ respectively. 
Since the symmetric functions $h_n(\xx)$ for $n\ge 1$ form a set of generators of the ring 
$\Lambda_{\QQ}$, we get as a corollary that the symmetric functions $\SP_{n}^{(1)}(\xx)$ and 
$\SP_{n}^{(2)}(\xx)$  also generate $\Lambda_{\QQ}$.
For a partition $\lambda \vdash n$ define 
$$\SP_{\lambda}^{(r)}(\xx):=\SP_{\lambda_1}^{(r)}(\xx)\SP_{\lambda_2}^{(r)}(\xx)\cdots 
\SP_{\lambda_{\ell(\lambda)}}^{(r)}(\xx).$$
\begin{theorem}\label{theorem:stirlingsymmetricfunctionsarebases}
For $n \ge 0 $, the sets $\{\SP_{\lambda}^{(1)}(\xx)\,\mid\, \lambda \vdash n\}$ and 
$\{\SP_{\lambda}^{(2)}(\xx)\,\mid\, \lambda \vdash n\}$ are bases 
for the $n$-th homogeneous graded component of $\Lambda_{\QQ}$.
\end{theorem}

\subsection{Poset (co)homology} \label{section:posetcohomology}
The symmetric functions $\SP^{(1)}_{n}(\xx)$ and $\SP^{(2)}_{n}(\xx)$ also appear in the context of 
poset topology.

A weighted partition of $[n]$ is a set $\{B_1^{\nu_1},B_2^{\nu_2},...,B_t^{\nu_t}\}$
where $\{B_1,B_2,...,B_t\}$ is a set partition of $[n]$ and $\nu_i \in \wcomp_{|B_i|-1}$ for all 
$i$. 
For $\nu,\mu \in \wcomp$ we say that $\nu\le \mu$ if $\nu(i)\le \mu(i)$ for every $i$.
 The {\it poset of weighted partitions} $\Pi_n^{w}$  is the set of weighted partitions of $[n]$ 
with order relation given by 
$\{A_1^{w_1},A_2^{w_2},...,A_s^{w_t}\}\le\{B_1^{v_1}, B_2^{v_2},...,B_t^{v_t}\}$ if the following
conditions hold:

\begin{itemize}
 \item $\{A_1,A_2,...,A_s\}$ is a refinement of $\{B_1,B_2,...,B_t\}$ and,
 \item If $B_j=A_{i_1}\cup A_{i_2}\cup ... \cup A_{i_l} $ then 
 $w_{i_1} + w_{i_2} + ... + w_{i_l} \le v_j $.
\end{itemize}

The poset $\Pi_n^{w}$ has one minimal element 
$\hat{0}:=\{1^{\mathbf{0}},2^{\mathbf{0}},\dots,n^{\mathbf{0}}\}$, where $\mathbf{0}=(0,0,\dots)$, 
and maximal elements $[n]^{\mu}:=\{[n]^{\mu}\}$ indexed by weak compositions $\mu \in 
\wcomp_{n-1}$. The following result was found in \cite{Dleon2013} using poset topology techniques 
(see \cite{Dleon2013} for all the definitions).

\begin{theorem}[\cite{Dleon2013}]\label{theorem:mobiuspartition}  For all $n \ge 1$,
\begin{align*}
\sum_{\mu \in
\wcomp_{n-1}}\mu_{\Pi_{n}^{w}}((\hat{0},[n]^{\mu}))\xx^{\mu}=(-1)^{n-1}{\SP}^{(2)}_{n-1}(\xx),
\end{align*}
where $\mu_{\Pi_{n}^{w}}((\hat{0},[n]^{\mu}))$ is the M\"obius invariant of the maximal interval 
$(\hat{0},[n]^{\mu})$.
\end{theorem}
It was also proved in \cite{Dleon2013} that there is an $\sym_n$-module isomorphism
 \begin{align*}  
 \lie(\mu) \simeq_{\sym_n}  \widetilde H^{n-3}((\hat 0, [n]^{\mu}))  \otimes \sgn_n 
\end{align*}
for all $\mu \in \wcomp_{n-1}$, where  $\lie(\mu)$ is the multilinear component determined by 
$\mu$ of the free Lie 
algebra with multiple compatible brackets, $\widetilde H^{n-3}((\hat 0, [n]^{\mu}))$ is the 
reduced cohomology of the interval $(\hat{0},[n]^{\mu})$ and $\sgn_n$ is the sign representation 
of the symmetric group $\sym_n$. The following theorem is a corollary of 
this isomorphism, Philip Hall's theorem, Theorem \ref{theorem:mobiuspartition} and the fact that 
$\Pi_{n}^{w}$ is Cohen-Macaulay (see \cite{Dleon2013}).

\begin{theorem}[\cite{Dleon2013}]\label{theorem:liemu}  For all $n \ge 1$,
\begin{align*}
\sum_{\mu \in
\wcomp_{n-1}}\dim \lie(\mu) \xx^{\mu}={\SP}^{(2)}_{n-1}(\xx),
\end{align*}
where $\dim V$ is the dimension of the vector space $V$.
\end{theorem}

It turns out that ${\SP}^{(1)}_{n-1}(\xx)$ also makes an appearance in the context of poset 
topology in subsequent work of the author \cite{Dleon2016}. 

Let 
$\B_n$ be the boolean algebra (the poset of subsets of $[n]$
ordered by inclusion) and $\WC_n$ be the poset formed by weak compositions $\mu$ such that 
$|\mu|\le n$ together with the order relation defined before.
Both posets are ranked and hence have well-defined poset maps $rk:\B_n\rightarrow C_{n+1}$ and
$rk:\WC_n\rightarrow C_{n+1}$ to the $n+1$ chain $C_{n+1}$. Recall that the \emph{Segre or fiber
product} $\displaystyle A\substack{\times \\f,g}B$ of
two poset maps $f:A \rightarrow C$ and $g:B \rightarrow C$ is the induced subposet of the product
$A \times B$ with elements $\{(a,b)\arrowvert \, f(a)=g(b)\}$. Denote by
$\displaystyle \B_n^w=\B_n\substack{\times \\rk,rk}\WC_n$, the poset of weighted subsets of $[n]$. 
This poset has one minimal element
$\hat{0}:=\varnothing^{\mathbf{0}}$ and  maximal elements $[n]^{\mu}:=\{[n]^{\mu}\}$ indexed by 
weak compositions $\mu \in 
\wcomp_{n}$.

The following theorem is proved in  \cite{Dleon2016}.

\begin{theorem}[\cite{Dleon2016}]\label{theorem:mobiusset}  For all $n \ge 0$,
\begin{align*}
\sum_{\mu \in
\wcomp_{n}}\mu_{\B_{n}^{w}}((\hat{0},[n]^{\mu}))\xx^{\mu}=(-1)^{n}{\SP}^{(1)}_{n}(\xx).
\end{align*}
\end{theorem}

There is a natural colored extension of the exterior algebra $\LL(V)$ over a 
vector space $V$ that provides an analogous version of Theorem \ref{theorem:liemu}. 
In this case we have that the multilinear components $\LL(\mu)$ of the colored exterior algebra 
satisfy the $\sym_n$-isomorphism
 \begin{align*}  
 \LL(\mu) \simeq_{\sym_n}  \widetilde H^{n-2}((\hat 0, [n]^{\mu}))
\end{align*}
for all $\mu \in \wcomp_n$, where  $[\hat 0, [n]^{\mu}]$ is the closed maximal interval in $\B_n^w$ 
determined by $\mu$.
\begin{theorem}[\cite{Dleon2016}]\label{theorem:exteriormu}  For all $n \ge 1$,
\begin{align*}
\sum_{\mu \in
\wcomp_{n}}\dim \LL(\mu) \xx^{\mu}={\SP}^{(1)}_{n}(\xx).
\end{align*}
\end{theorem}

\begin{remark}
 Both Theorem \ref{theorem:exponentialcompositional} and 
Theorem \ref{theorem:exponentialmultiplicative} can also be derived from Theorem 
\ref{theorem:mobiuspartition} and \ref{theorem:mobiusset} using the recursive definition of the 
M\"obius invariant, see \cite{Dleon2013,Dleon2016}.
\end{remark}

\subsection{Stable $n$-pointed curves}
In \cite{KaufmannManinZagier1996}  another surprising connection with the
symmetric functions $\SP^{(2)}_{n}(\xx)$ appeared in the context of moduli spaces 
$\overline{M_{0,n}}$ of 
stable $n$-pointed curves of
genus $0$. See \cite{KaufmannManinZagier1996} for the proper definitions and notation. Let
$\omega_n(i) \in H^{2i}(\overline{M_{0,n}},\QQ)$ denote the Mumford classes and for a partition
$\lambda=1^{m_1}2^{m_2}\dots$ denote $\omega^{\lambda}_n=\prod_i\omega_n(i)^{m_i}$. The
\emph{higher Weil-Petersson volumes} are defined as
\begin{align*}
WP(\lambda)=\int_{\overline{M_{0,n}}} \omega^{\lambda}_n.
\end{align*}
The following theorem follows directly from Theorem \ref{theorem:exponentialcompositional} and
\cite[Theorem 2.4]{KaufmannManinZagier1996} after making the identifications 
$s_k\mapsto -\frac{p_k(\xx)}{k}$ and $y\mapsto -y$
($s_k$ a formal variable and $p_k(\xx)$ the power sum symmetric function) and writting down the
proper definitions.
\begin{theorem}\label{theorem:modulispaces}For $n\ge 0$
\begin{align*}
 \SP^{(2)}_{n}(\xx) = \sum_{\lambda \vdash n} (-1)^{n-1-l(\lambda)} WP(\lambda) 
\dfrac{p_{\lambda}(\xx)}{z_{\lambda}}
\end{align*}
\end{theorem}

In \cite{KaufmannManinZagier1996} the authors also provide a recursive definition, a differential 
equation 
and the following closed formula
for the coefficients $WP(\lambda)$. For partitions $\nu^1,\nu^2\dots,\nu^k$ denote by
$\nu^1+\nu^2+\cdots+\nu^k$ the partition obtained by taking the union of all the parts
of all the $\nu^i$.

\begin{theorem}[\cite{KaufmannManinZagier1996} Corollary 2.3]
For $n\ge 0$ and $\lambda\vdash n$
\begin{align*} 
WP(\lambda)=n!\sum_{k=0}^{l(\lambda)}(-1)^{l(\lambda)-k}\binom{n+k}{k} 
\sum_{\substack{\nu^1,\dots,\nu^k \\
\nu^1+\nu^2+\cdots+\nu^k=\lambda \\ \nu^i \neq
0}}\dfrac{\prod_{j=1}^{l(\lambda)}\binom{m_j(\lambda)}{m_j(\nu^1),m_j(\nu^2)\dots,m_j(\nu^k)}}{
\prod_ { i=1} ^ {k} (|\nu^i|+1)!}.
\end{align*}
\end{theorem}

\subsection{Open questions}
Theorems \ref{theorem:exponentialmultiplicative}  and 
\ref{theorem:exponentialcompositional} say that the
symmetric functions $\SP^{(1)}_{n}(\xx)$  and $\SP^{(2)}_{n}(\xx)$ are involved in the computation 
of 
multiplicative and compositional inverses of power series. According to Theorem 
\ref{theorem:stirlingsymmetricfunctionsarebases} these functions are generators of 
$\Lambda_{\QQ}$. They are also the generating 
functions for the M\"obius invariants of the maximal intervals of the posets $B_n^w$ (Theorem 
\ref{theorem:mobiusset}) and $\Pi_n^w$ (Theorem \ref{theorem:mobiuspartition}), respectively. They 
are the dimension generating functions  of the 
multilinear components of the free colored exterior algebra (Theorem \ref{theorem:exteriormu}) and 
the free multibracketed Lie algebra 
(Theorem \ref{theorem:liemu}). Finally, $\SP^{(2)}_{n}(\xx)$ is the 
generating function (viewed in the $p$ basis expansion) of the generalized Weil-Petersson volumes 
of the moduli space $\overline{M_{0,n}}$  (Theorem \ref{theorem:modulispaces}). The question now is 
to 
understand whether some or all of these results extend to the more general family of 
symmetric functions $\SP^{(r)}_{n}(\xx)$ for $r \ge 3$.

\begin{question}
Is there a more general family of posets $P(n,r)$ such that the weighted generating function for
the M\"obius invariant of maximal intervals is given up to sign by $\SP^{(r)}_{n}(\xx)$ as in
Theorems \ref{theorem:mobiusset} and \ref{theorem:mobiuspartition}?
\end{question}

\begin{question}
Is there any combinatorial context where the functions $\SP^{(r)}_{n}(\xx)$ are meaningful for
any $r \ge 3$? Are there formulas similar to equations 
(\ref{equation:exponentialmultiplicativeinverse}) and 
(\ref{equation:exponentialcompositionalinverse})
for the $\SP^{(r)}_{n}(\xx)$ when $r\ge3$?
\end{question}

\begin{question}
Are the sets $\{\SP_{\lambda}^{(r)}(\xx)\,\mid\, \lambda \vdash n\}$ bases 
for the $n$-th homogeneous graded component of $\Lambda_{\QQ}$ for every $r\ge 3$?
\end{question}

As of the time of the construction of this article we do not know of any partial result in any of 
these directions. One central question is whether the family of $r$-Stirling permutations for 
$r\ge 3$ is the right family of multipermutations to extend the results in this work or if 
another family of multipermutations generalizing both $\sym_n$ and $\Q_n$ is needed.

\subsection{Further work}
Theorem \ref{theorem:exponentialcompositional} can be generalized in a different direction by 
considering families of normalized $k$-ary trees that satisfy a certain coloring condition. This 
generalization however does not include Theorem \ref{theorem:exponentialmultiplicative} as a 
special case. We present and explore this generalization in a future article.

\section*{Acknowledgments}
The author is grateful to Michelle Wachs for her guidance during this project and all the 
valuable discussions. The author also would like to thank Ira Gessel and Fran\c{c}ois 
Bergeron for the valuable discussions and useful comments to improve both the context and the 
presentation of the results. The author also appreciates the very helpful comments of two anonymous 
referees.
A great part of this work was done while the author was supported by NSF Grant  DMS 1202755 and by 
Colciencias (Departamento Administrativo de Ciencia, Tecnolog\'ia e Innovaci\'on).

\bibliographystyle{abbrv}
\bibliography{stirlingsym}

\end{document}